\documentclass[11pt]{amsart}
\usepackage{amsfonts}%para algunos tipos de letra, p.e.R(reales)%
\usepackage{amsmath}
\usepackage{graphicx}
\usepackage[english]{babel}

\newtheorem{theorem}{Theorem}

\newtheorem{corollary}{Corollary}

\newtheorem{definition}{Definition}

\newtheorem{lemma}{Lemma}

\newtheorem{proposition}{Proposition}
\newtheorem{remark}{Remark}

\newcommand{\Sum}{\displaystyle\sum}%_{#1}^{#2}}
\newcommand{\Lim}{\displaystyle\lim}%_{#1}^{#2}}

\newcommand{\Frac}[2]{\displaystyle\frac{#1}{#2}}
\newcommand{\Prod}{\displaystyle\prod}%_{#1}^{#2}}

\begin{document}

\title{Almost Periodic Functions in terms of Bohr's Equivalence Relation$^\ast$}
\footnotetext{$^\ast$This paper is a modified version of the article ``Almost periodic functions in terms of Bohr's equivalence relation" (Ramanujan J., 2017, https://doi.org/10.1007/s11139-017-9950-1). As we pointed out in our corrigendum (``Correction to: Almost periodic functions in terms of Bohr's equivalence relation", Ramanujan J., 2019, https://doi.org/10.1007/s11139-019-00150-3), we correct a gap found in the previous version of 2017. This gap is due to a property that does not generally hold. However, under the additional assumption of existence of an integral basis, it does hold. Furthermore, the equivalence relation which is inspired by that of Bohr is revised to adapt correctly the situation in the general case. The structure of the paper (including numbering) is the same as that of the published version. Moreover, the proofs of propositions \ref{DefEquiv} %(now, the vector $\mathbf{x_n}$ in $6^{10}$ is well-defined)
and \ref{prop} %($\{f_l\}_{l\geq 1}$ contains a subsequence converging to $h$ because of Helly's selection principle)
have been written in more detail.
}

\author{J.M. Sepulcre}
\address{Department of Mathematics\\ University of
Alicante, 03080-Alicante\\
Spain} \email{JM.Sepulcre@ua.es}

\author{T. Vidal}
\address{University of
Alicante, 03080-Alicante\\
Spain} \email{tmvg@alu.ua.es}

%\subjclass[2010]{Primary: 30D20, 11J72, 11K60}




\begin{abstract}
In this paper we introduce an equivalence relation on the classes of almost periodic functions of a real or complex variable which is used to refine Bochner's result that characterizes these spaces of functions. In fact, with respect to the topology of uniform
convergence, we prove that the limit points of the family of translates of an almost periodic function are precisely the functions which are equiva\-lent to it, which leads us to a characterization of almost periodicity. In particular we show that any exponential sum which is equivalent to the Riemann zeta function, $\zeta(s)$, can
be uniformly approximated in $\{s=\sigma+it:\sigma>1\}$ by certain vertical translates of $\zeta(s)$.
%\keywords{Almost periodic functions \and Exponential sums \and Riemann zeta function  \and Bochner's theorem}
% \subclass{30D20 \and 30B50 \and 11K60 \and 30Axx}
\end{abstract}

\maketitle

\section{Introduction}

The theory of almost periodic functions, which was created and developed in its main features by H. Bohr during the $1920$'s, opened a way to study a wide class of trigonometric series of the general type and even exponential series (see for example \cite{Besi,Bohr,Bohr2,Corduneanu1,Corduneanu,Jessen}). This theory shortly acquired numerous applications to various areas of mathematics, from harmonic analysis to differential equations.

In the case of the functions that are defined on the real numbers, the notion of almost periodicity leads us to generalize purely periodic functions. Let $f(t)$ be a real or complex function of an unrestricted real variable $t$. In this paper, in order to be almost periodic, $f(t)$ must be continuous, and for every $\varepsilon>0$ there corresponds a number $l=l(\varepsilon)>0$ such that each interval of length $l$ contains a number $\tau$ satisfying
$|f(t+\tau)-f(t)|\leq\varepsilon$ for all $t$. As in \cite{Corduneanu}, we will denote as $AP(\mathbb{R},\mathbb{C})$ the space of almost periodic functions in the sense of this definition (Bohr's condition), which coincides with the notion of uniform almost periodicity used in \cite{Besi}. As in classical Fourier analysis, every almost periodic function is bounded and is associated with a Fourier series with real frequencies.

A very important result of this
theory is the approximation theorem according to which the class of almost
periodic functions $AP(\mathbb{R},\mathbb{C})$ is identical with the class of those functions which can be
approximated uniformly by trigonometric polynomials of the type
$$a_1e^{i\lambda_1t}+\ldots+a_ne^{i\lambda_nt}$$
with arbitrary real exponents $\lambda_j$ and arbitrary complex coefficients $a_j$. Moreover,
%in the context of the space of the bounded and continuous functions on $\mathbb{R}$,
S. Bochner observed that Bohr's notion of almost periodicity of a function $f$ is equivalent to the relative compactness, in the sense of uniform convergence, of the family of its translates $\{f(t+h)\}$, $h\in\mathbb{R}$. So, in later literature, some authors defined almost periodic functions in this way (e.g. see \cite{Bochner,Corduneanu,Fink}).

Now we focus our attention on the theory of the almost periodic functions of a complex variable, which was theorized in \cite{BohrAnalytic} (see also \cite{Besi,Bohr,Corduneanu1,Favorov1,Jessen}).
Let $f(s)$ be a function of a complex variable $s=\sigma+it$ which is analytic in a vertical strip $U=\{s=\sigma+it:\alpha<\sigma<\beta\}$ ($-\infty\leq\alpha<\beta\leq\infty$). By analogy with the case of a real variable, it is called almost periodic in $U$ if for any
 $\varepsilon>0$ and every reduced strip $U_1=\{s=\sigma+it:\sigma_1\leq\sigma\leq\sigma_2\}$ of $U$ there exists a number $l=l(\varepsilon)>0$ such that each interval of length $l$ contains a number $\tau$ satisfying the inequality
$|f(s+i\tau)-f(s)|\leq \varepsilon$ for $s$ in $U_1$.
This definition implies in particular that, for any fixed $\sigma\in(\alpha,\beta)$, the function $h_{\sigma}(t):=f(\sigma+it)$ is an almost periodic function of the real variable $t$. Moreover, the requirement above implies that the almost periodicity should take place \textit{uniformly} on the various straight lines.
In addition, the Fourier series of these functions are obtainable from a certain exponential series of the form $\sum_{n\geq 1} a_ne^{\lambda_n s}$ with complex coefficients $a_n$ and real exponents $\lambda_n$. This associated series is called the Dirichlet series of the given almost periodic function (see \cite[p.147]{Besi}, \cite[p.77]{Corduneanu1} or \cite[p.312]{Jessen}). % determines the function uniquely. %and, for any fixed $\sigma\in (\alpha,\beta)$, gives the Fourier series of $F_{\sigma}(t)$.

The space of almost periodic functions in a vertical strip $U\subset \mathbb{C}$, which will be denoted as $AP(U,\mathbb{C})$, coincides with the set of the functions which can be approximated uniformly in every reduced strip %$\{s=\sigma+it:\sigma_1\leq\sigma\leq\sigma_2\}$
of $U$ by exponential polynomials with complex coefficients and real exponents (see \cite[Theorem 3.18]{Corduneanu1}). These aproximating exponential polynomials can be found by Bochner-Fej\'{e}r's summation (see, in this regard, \cite[Chapter 1, Section 9]{Besi}).
% S. Bochner observed that Bohr's notion of almost periodicity of a function $f$ in a vertical strip U is equivalent to that every sequence {f(s + it n )}, t n\in\mathbb{R}, of vertical translations of $f$ has a subsequence that converges uniformly for $s$ in $U$.
%On the other hand, Bochner proved that the almost periodicity of the analytic function $f(s)$ in the strip $U$ means that the set of functions $f(s+ih)$, for real $h$, is relatively compact in the sense of uniform convergence in that strip.
%Bochner observed first that the function f is almost periodic if and only if its set of translates has compact closure in the metric of uniform convergence. This definition was much easier to work with than Bohr's definition
In the same manner, Bohr's notion of almost periodicity of a function $f(s)$ in a vertical strip $U$ is equivalent to the relative compactness of the set of its vertical translates, $\{f(s+ih)\}$, $h\in\mathbb{R}$, with the topology of the uniform convergence on reduced strips. %in the sense of uniform convergence in that strip. %In fact, Bochner called this property of functions normality

On the other hand, we also recall that the class of general Dirichlet series consists of the series that take
the form
$\Sum_{n\geq 1}a_ne^{-\lambda_n s},\ a_n\in\mathbb{C},$
 where
$\{\lambda_n\}$ is a strictly increasing sequence of posi\-tive numbers
tending to infinity.
%-%So, by taking $\lambda_n=\log n$, $n\in\mathbb{N}$, we obtain the so-called
%ordinary Dirichlet series. %, which are thus of the form $\Sum_{n\geq 1}\frac{a_n}{n^s}.$
%The Riemann zeta function $\zeta(s)$ is a function of a complex variable $s=\sigma+it$ which plays a pivotal role
%in analytic number theory.
%It is defined as the analytic continuation of the function defined for $\sigma>1$ by the sum
%$$\zeta(s)=\Sum_{n=1}^{\infty}\Frac{1}{n^s}.$$
In particular, a classical ordinary Dirichlet series is the Riemann zeta
function, given by
$\zeta(s):=\Sum_{n\geq 1}\frac{1}{n^s},$
that converges absolutely in
$\{s\in\mathbb{C}:\operatorname{Re}s>1\}$ and admits an analytic
continuation over the whole complex plane with only a simple pole
at $s=1$. A remarkable property of $\zeta(s)$, which is satisfied in a certain location on its critical strip, is called universality, or Voronin's universality theorem, and it states, roughly speaking, that any non-vanishing analytic function can be uniformly
approximated by certain shifts of the Riemann zeta-function. A similar universality property has been shown for other functions, such as the Dirichlet $L$-functions. For a complete study on this very interesting property, we suggest \cite{KaraVoro,Lauri,Lauri2}.

Regarding general Dirichlet series, H. Bohr introduced an equivalence relation among them that led to exceptional results (see for example Bohr's equivalence theorem in \cite{Apostol}).
By analogy with Bohr's theory, in this paper we establish an equi\-valence relation on the classes of exponential sums of the form
\begin{equation*}%\label{eqq}
\sum_{j\geq 1}a_je^{\lambda_js},\ a_j\in\mathbb{C},\ \lambda_j\in\Lambda,
\end{equation*}
where $\Lambda=\{\lambda_1,\ldots,\lambda_j,\ldots\}$ is an arbitrary countable set of distinct real numbers (not necessarily unbounded), and we extend it to the context of almost periodic functions.
In this way, with respect to the topology of uniform
convergence, the main result of our paper shows that, fixed an almost periodic function, the limit points of the set of its translates are precisely the functions which are equiva\-lent to it (see theorems \ref{mth0} and \ref{mth} in this paper for the real and complex case respectively).
This means that Bochner's result is now refined in the sense that we show that the condition of almost periodicity is equivalent to that every sequence of translates has a subsequence that converges uniformly to an equivalent function (see theorems \ref{mthh} and \ref{mthhcc} in this paper).

In particular, in terms of Voronin's universality theorem,
we show that any exponential sum which is equivalent to the Riemann zeta function can
be uniformly approximated in $\{s=\sigma+it:\sigma>1\}$ by certain vertical translates of the Riemann zeta-function (see Theorem \ref{mtrzf} in this paper). For example, we assure the
existence of two increasing unbounded sequences $\{\tau_n\}_{n\geq 1}$ and $\{\varsigma_n\}_{n\geq 1}$
of positive numbers such that the sequences of functions
$\{\zeta(s+i\tau_n)\}$ and $\{\zeta(s+i\varsigma_n)\}$, $n\in\mathbb{N}$, converge uniformly to
$\zeta(s)$ and $\zeta_{\lambda}(s)$ on every reduced strip of
$\{s\in\mathbb{C}:\operatorname{Re}s>1\}$ respectively, where $\zeta_{\lambda}(s):=\sum_{n\geq 1}\frac{\lambda(n)}{n^s}$ is the ordinary Dirichlet series for the Liouville function $\lambda(n)$.
To the best of our knowledge, these results
have not been considered in the literature.
Finally, as a
consequence, we will obtain an alternative
demonstration of a known result related to the infimum of
$|\zeta(s)|$ on certain regions in the half-plane $\sigma\geq 1$
(see corollaries \ref{positive} and \ref{ult} in this paper).

\section{The exponential sums of the classes $\mathcal{S}_{\Lambda}$ and Bohr's equivalence relation on them}

%Based on Bohr's equivalence relation for general Dirichlet series (see for example \cite[p.173]{Apostol}),
Consider the following equivalence relation which constitutes our starting point.
\begin{definition}\label{DefEquiv0}
Let $\Lambda$ be an arbitrary countable subset of distinct real numbers, $V$ the $\mathbb{Q}$-vector space generated by $\Lambda$ ($V\subset \mathbb{R}$), and $\mathcal{F}$
the $\mathbb{C}$-vector space of arbitrary functions $\Lambda\to\mathbb{C}$. %("Fourier coefficients").
We define
a relation $\sim$ on $\mathcal{F}$ by $a\sim b$ if there exists a $\mathbb{Q}$-linear map $\psi:V\to\mathbb{R}$ such that
$$b(\lambda)=a(\lambda)e^{i\psi(\lambda)},\ (\lambda\in\Lambda).$$
\end{definition}

It is immediate that $\sim$ is an equivalence relation. The reader can observe that this equivalence relation is based on that of \cite[p.173]{Apostol} which was defined for general Dirichlet series and was characterized in terms of a completely multiplicative function \cite[Theorem 8.12]{Apostol}. Bohr used it in that case in order to get so-called Bohr's equivalence theorem.

\begin{remark}\label{unodrx}
With respect to the relation $\sim$ on $\mathcal{F}$, we note the following easy facts to clarify and simplify later proofs:
\begin{itemize}
\item[i)] $a\sim b$ implies that $ac\sim bc$.
\item[ii)] If $\Lambda'\subset \Lambda$ and functions $a',b'$ on $\Lambda'$ are extended to $a,b$ on $\Lambda$ by
$$a(\lambda)=a'(\lambda),\ (\lambda\in\Lambda')$$
$$b(\lambda)=b'(\lambda),\ (\lambda\in\Lambda')$$
$$a(\lambda)=b(\lambda)=0,\ (\lambda\in\Lambda\setminus\Lambda'),$$
then $a\sim b$ if and only if $a'\sim b'$.
\end{itemize}
\end{remark}

In this paper we will use Definition \ref{DefEquiv0} to introduce an equivalence relation on certain subclasses of exponential sums, which will be referred from now on as
 expressions of the type
$$P_1(p)e^{\lambda_1p}+\ldots+P_j(p)e^{\lambda_jp}+\ldots,$$
where the frequencies $\lambda_j$ are complex numbers and the $P_j(p)$ are polynomials in $p$.  What is more, we will consider some functions which are associated with a concrete
subclass of these exponential sums, where the parameter $p$ will be changed by $s=\sigma+it$ in the complex case, or by $t$ in the real case.

\begin{definition}
Let $\Lambda=\{\lambda_1,\lambda_2,\ldots,\lambda_j,\ldots\}$ be an arbitrary countable set of distinct real numbers, which we will call a set of exponents or frequencies. We will say that an exponential sum is in the class $\mathcal{S}_{\Lambda}$ if it is a formal series of type
 \begin{equation}\label{eqqnew}
\sum_{j\geq 1}a_je^{\lambda_jp},\ a_j\in\mathbb{C},\ \lambda_j\in\Lambda.
\end{equation}
Also, we will say that $a_1,a_2,\ldots,a_j,\ldots$ are the coefficients of this exponential sum.
\end{definition}

It is clear that expression (\ref{eqqnew}) is not necessarily associated with a convergent series that defines an holomorphic function. Precisely, we call it formal series to distinguish it from an ordinary
series. In fact, in the theory of formal series, the parameter $p$ of (\ref{eqqnew}) is never
assigned a numerical value and questions of convergence or divergence are
irrelevant. In this respect, we operate on formal series algebraically as though they were
convergent series and the expression $e^{\lambda_jp}$ is simply a device for
locating the position of the $j$th coefficient $a_j$. In this way, if $A_1(p)=\sum_{j\geq 1}a_je^{\lambda_jp}$ and $A_2(p)=\sum_{j\geq 1}b_je^{\lambda_jp}$ are two formal series in $\mathcal{S}_{\Lambda}$, then $A_1(p)=A_2(p)$ means that $a_j=b_j$ for each $j\geq 1$. Moreover, $A_1(p)+A_2(p):=\sum_{j\geq 1}(a_j+b_j)e^{\lambda_jp}$ and $A_1(p+h):=\sum_{j\geq 1}a_je^{\lambda_jh}e^{\lambda_jp}$ for all $h\in\mathbb{C}$.

Based on Definition \ref{DefEquiv0}, we next consider the following equivalence relation on the classes $\mathcal{S}_{\Lambda}$. In fact, we will say that two exponential sums in $\mathcal{S}_{\Lambda}$ are equivalent when their coefficients adapt to Definition \ref{DefEquiv0} in the following sense.

\begin{definition}[mod.]\label{DefEquiv00}
Given $\Lambda=\{\lambda_1,\lambda_2,\ldots,\lambda_j,\ldots\}$ a set of exponents, consider $A_1(p)$ and $A_2(p)$ two exponential sums in the class $\mathcal{S}_{\Lambda}$, say
$A_1(p)=\sum_{j\geq1}a_je^{\lambda_jp}$ and $A_2(p)=\sum_{j\geq1}b_je^{\lambda_jp}.$
We will say that $A_1$ is equivalent to $A_2$ if for each integer value $n\geq1$, with $n\leq\sharp\Lambda$, it is satisfied
$a_{n}^*\sim b_{n}^*$, where $a_{n}^*,b_{n}^*:\{\lambda_1,\lambda_2,\ldots,\lambda_{n}\}\to\mathbb{C}$ are the functions given by $a_{n}^*(\lambda_j):=a_j$ y $b_{n}^*(\lambda_j):=b_j$, $j=1,2,\ldots,n$ and $\sim$ is in Definition \ref{DefEquiv0}.
%Given $\Lambda=\{\lambda_1,\lambda_2,\ldots,\lambda_j,\ldots\}$ a set of exponents, consider $A_1(p)$ and $A_2(p)$ two exponential sums in the class $\mathcal{S}_{\Lambda}$, say
%$A_1(p)=\sum_{j\geq1}a_je^{\lambda_jp}$ and $A_2(p)=\sum_{j\geq1}b_je^{\lambda_jp}.$
%We will say that $A_1$ is equivalent to $A_2$ if $a\sim b$, where $a,b:\Lambda\to\mathbb{C}$ are the functions given by $a(\lambda_j):=a_j$ y $b(\lambda_j):=b_j$, $j=1,2,\ldots$ and $\sim$ is in Definition \ref{DefEquiv0}.
%Fixed a basis $G_{\Lambda}$ for $\Lambda$, for each $j\geq1$ let $\mathbf{r}_j$ be the vector of rational components verifying (\ref{errej}).  %-%$$\lambda_j=<\mathbf{r}_j,\mathbf{g}>=\sum_{k=1}^{q_j}r_{j,k}g_k,$$ where
%  %$\mathbf{g}=(g_1,g_2,\ldots,g_k,\ldots)$ is the vector of the elements of a basis $G_{\Lambda}$ for $\Lambda$.
%Thus $A_1\sim A_2$ %, relative to the basis $G_{\Lambda}$,
%if and only if there exists $\mathbf{x}_0=(x_{0,1},x_{0,2},\ldots,x_{0,k},\ldots)\in \mathbb{R}^{\sharp G_{\Lambda}}$
%such that $b_j=a_j e^{<\mathbf{r}_j,\mathbf{x}_0>i}$ for every $j\geq 1$. %In that case, we will write $A_1\sim A_2$.
\end{definition}

We will use $\sim$ for the equivalence relation introduced in Definition \ref{DefEquiv0} and $\shortstack{$_{{\fontsize{6}{7}\selectfont *}}$\\$\sim$}$ for that of Definition \ref{DefEquiv00}.

Let
$G_{\Lambda}=\{g_1, g_2,\ldots, g_k,\ldots\}$ be a basis of the
$\mathbb{Q}$-vector space generated by a set $\Lambda=\{\lambda_1,\lambda_2,\ldots\}$ of exponents, %group $W=\mathbb{Z}w_1+\mathbb{Z}w_2+\ldots+\mathbb{Z}w_j+\ldots$,
which implies that $G_{\Lambda}$ is linearly independent over the rational numbers and each $\lambda_j$ is expressible as a finite linear combination of terms of $G_{\Lambda}$, say
\begin{equation}\label{errej}
\lambda_j=\sum_{k=1}^{i_j}r_{j,k}g_k,\ \mbox{for some }r_{j,k}\in\mathbb{Q},\ i_j\in\mathbb{N}.
\end{equation}
In this paper, by abuse of language, we will also say that $G_{\Lambda}$ is a basis for $\Lambda$. Moreover, we will say that $G_{\Lambda}$ is an \textit{integral basis} for $\Lambda$ when $r_{j,k}\in\mathbb{Z}$ for any $j,k$.
Now, by taking this
into account, we next show that the equivalence relation introduced in Definition \ref{DefEquiv00} on the classes $\mathcal{S}_{\Lambda}$ can be characterized in terms of a basis for $\Lambda$. %Although the proof of this result is reasonably simple, we include it for the sake of completeness.

\begin{proposition}[mod.]\label{DefEquiv}
Given $\Lambda=\{\lambda_1,\lambda_2,\ldots,\lambda_j,\ldots\}$ a set of exponents, consider $A_1(p)$ and $A_2(p)$ two exponential sums in the class $\mathcal{S}_{\Lambda}$, say
$A_1(p)=\sum_{j\geq1}a_je^{\lambda_jp}$ and $A_2(p)=\sum_{j\geq1}b_je^{\lambda_jp}.$
Fixed a basis $G_{\Lambda}$ for $\Lambda$, for each $j\geq1$ let $\mathbf{r}_j\in \mathbb{R}^{\sharp G_{\Lambda}}$ be the vector of rational components verifying (\ref{errej}).  %-%$$\lambda_j=<\mathbf{r}_j,\mathbf{g}>=\sum_{k=1}^{q_j}r_{j,k}g_k,$$ where
  %$\mathbf{g}=(g_1,g_2,\ldots,g_k,\ldots)$ is the vector of the elements of a basis $G_{\Lambda}$ for $\Lambda$.
Then $A_1\ \shortstack{$_{{\fontsize{6}{7}\selectfont *}}$\\$\sim$}\ A_2$ %, relative to the basis $G_{\Lambda}$,
if and only if for each integer value $n\geq1$, with $n\leq\sharp\Lambda$, there exists $\mathbf{x}_n=(x_{n,1},x_{n,2},\ldots,x_{n,k},\ldots)\in \mathbb{R}^{\sharp G_{\Lambda}}$
such that $b_j=a_j e^{<\mathbf{r}_j,\mathbf{x}_n>i}$ for $j=1,2,\ldots,n$.\vspace{0.1cm} %In that case, we will write $A_1\sim A_2$.

\noindent Furthermore, if $G_{\Lambda}$ is an integral basis for $\Lambda$ then $A_1\ \shortstack{$_{{\fontsize{6}{7}\selectfont *}}$\\$\sim$}\ A_2$ %, relative to the basis $G_{\Lambda}$,
if and only if there exists $\mathbf{x}_0=(x_{0,1},x_{0,2},\ldots,x_{0,k},\ldots)\in \mathbb{R}^{\sharp G_{\Lambda}}$
such that $b_j=a_j e^{<\mathbf{r}_j,\mathbf{x}_0>i}$ for every $j\geq 1$. %In that case, we will write $A_1\sim A_2$.
\end{proposition}
\begin{proof}
For each integer value $n\geq 1$, let $V_n$ be the $\mathbb{Q}$-vector space generated by $\{\lambda_1,\ldots,\lambda_n\}$, $V$ the $\mathbb{Q}$-vector space generated by $\Lambda$, and
$G_{\Lambda}=\{g_1, g_2,\ldots, g_k,\ldots\}$ a basis of $V$.
If $A_1\ \shortstack{$_{{\fontsize{6}{7}\selectfont *}}$\\$\sim$}\ A_2$, by Definition \ref{DefEquiv00} for each integer value $n\geq 1$, with $n\leq\sharp\Lambda$, there exists a $\mathbb{Q}$-linear map $\psi_n:V_n\to\mathbb{R}$ such that
$b_j=a_je^{i\psi_n(\lambda_j)},\ j=1,2\ldots,n.$
Hence
$b_j=a_je^{i\sum_{k=1}^{i_j}r_{j,k}\psi_n(g_k)},\ j=1,2\ldots,n$
or, equivalently,
$b_j=a_j e^{i<\mathbf{r}_j,\mathbf{x}_n>},\ j=1,2\ldots,n,$
with $\mathbf{x}_n:=(\psi_n(g_1),\psi_n(g_2),\ldots,\psi_n(g_p),0,\ldots)$, where $p=\max\{i_1,i_2,\ldots,i_n\}$. %<--modificado a partir del referee Carpathian (para Arxiv)
Conversely, suppose the existence, for each integer value $n\geq 1$, of a vector $\mathbf{x}_n=(x_{n,1},x_{n,2},\ldots,x_{n,k},\ldots)\in \mathbb{R}^{\sharp G_{\Lambda}}$
such that $b_j=a_j e^{<\mathbf{r}_j,\mathbf{x}_n>i}$, $j=1,2\ldots,n$. Thus a $\mathbb{Q}$-linear map $\psi_n:V_n\to\mathbb{R}$ can be defined from $\psi_n(g_k):=x_{n,k}$, $k\geq 1$. Therefore $\psi_n(\lambda_j)=\sum_{k=1}^{i_j}r_{j,k}\psi(g_k)=$ $<\mathbf{r}_j,\mathbf{x}_n>,\ j=1,2\ldots,n,$
and the result follows.

Now, suppose that $G_{\Lambda}$ is an integral basis for $\Lambda$ and $A_1\ \shortstack{$_{{\fontsize{6}{7}\selectfont *}}$\\$\sim$}\ A_2$.
Thus, by above, for each fixed integer value $n\geq 1$, let $\mathbf{x}_n=(x_{n,1},x_{n,2},\ldots)\in\mathbb{R}^{\sharp G_{\Lambda}}$ be a vector such that
$b_j=a_j e^{i<\mathbf{r}_j,\mathbf{x}_n>},\ j=1,2\ldots,n.$ Since each component of $\mathbf{r}_j$ is an integer number, without loss of generality, we can take $\mathbf{x}_n\in[0,2\pi)^{\sharp G_{\Lambda}}$ as the unique vector in $[0,2\pi)^{\sharp G_{\Lambda}}$ satisfying the above equalities, where we assume $x_{n,k}=0$ for any $k$ such that $r_{j,k}=0$ for $j=1,\ldots,n$. Therefore, under this assumption, if $m>n$ then $x_{m,k}=x_{n,k}$ for any $k$ so that $x_{n,k}\neq 0$. In this way, we can construct a vector $\mathbf{x}_0=(x_{0,1},x_{0,2},\ldots,x_{0,k},\ldots)\in [0,2\pi)^{\sharp G_{\Lambda}}$ such that $b_j=a_j e^{<\mathbf{r}_j,\mathbf{x}_0>i}$ for every $j\geq 1$. Indeed, if $r_{1,k}\neq 0$ then the component $x_{0,k}$ is chosen as $x_{1,k}$, and if $r_{1,k}=0$ then
each component $x_{0,k}$ is defined as $x_{n+1,k}$ where $r_{j,k}=0$ for $j=1,\ldots,n$ and $r_{n+1,k}\neq 0$.
%
%Now, by reductio ad absurdum, suppose that there does not exist a vector $\mathbf{x}_0=(x_{0,1},x_{0,2},\ldots,x_{0,k},\ldots)\in [0,2\pi)^{\sharp G_{\Lambda}}$
%such that $b_j=a_j e^{<\mathbf{r}_j,\mathbf{x}_0>i}$ for every $j\geq 1$. This implies that there exist $n,k\in \mathbb{N}$ such that $b_n=a_n e^{<\mathbf{r}_n,\mathbf{x}_n>i}$, $b_{n+1}=a_{n+1} e^{<\mathbf{r}_{n+1},\mathbf{x}_{n+1}>i}$ and $x_{n,k}\neq x_{n+1,k}$.
Conversely, if there exists $\mathbf{x}_0=(x_{0,1},x_{0,2},\ldots,x_{0,k},\ldots)\in \mathbb{R}^{\sharp G_{\Lambda}}$
such that $b_j=a_j e^{<\mathbf{r}_j,\mathbf{x}_0>i}$ for every $j\geq 1$, then it is clear that $A_1\ \shortstack{$_{{\fontsize{6}{7}\selectfont *}}$\\$\sim$}\ A_2$ under Definition \ref{DefEquiv00}.
\end{proof}

\vspace{-0.1cm}If $\Lambda$ admits an integral basis, note that the set of all exponential sums $A(p)$ in an equivalence class $\mathcal{G}$ in $\mathcal{S}_{\Lambda}/\shortstack{$_{{\fontsize{6}{7}\selectfont *}}$\\$\sim$}$ can be determined by a function $E_{\mathcal{G}}:\mathbb{R}^{\sharp G_{\Lambda}}\rightarrow \mathcal{S}_{\Lambda}$ of the form
\begin{equation*}\label{2.4.000}
E_{\mathcal{G}}(\mathbf{x}):=\sum_{j\geq1}a_je^{<\mathbf{r}_j,\mathbf{x}>i}e^{\lambda_jp}\text{, }
\mathbf{x}=(x_1,x_2,\ldots,x_k,\ldots)\in%
\mathbb{R}^{\sharp G_{\Lambda}},
\end{equation*}%
where $a_1,a_2,\ldots a_j,\ldots$ are the coefficients of an exponential sum in $\mathcal{G}$ and the $\mathbf{r}_j$'s are the vectors of integer components associated with a prefixed integral basis $G_{\Lambda}$ for $\Lambda$.

%\begin{remark}\label{equiv2}
%In addition, Proposition \ref{DefEquiv} can be easily formulated, in an equivalent way, in terms of a product. Indeed,
%given $\Lambda=\{\lambda_1,\lambda_2,\ldots,\lambda_j,\ldots\}$ a set of exponents, let
%$A_1(p)=\sum_{j\geq 1}a_je^{\lambda_jp}$ and $A_2(p)=\sum_{j\geq 1}b_je^{\lambda_jp}$
%be two exponential sums in the class $\mathcal{S}_{\Lambda}$.
%Thus $A_1$ and $A_2$ are equivalent if and only if there exists a countable set $\{x_{0,1},x_{0,2},\ldots\}$ of real numbers such that $b_j=a_j e^{<\mathbf{r}_j,\mathbf{x}_0>i}$, $j\geq 1$, which is equivalent to $b_j=a_j \prod_{k=1}^{q_j}e^{ir_{j,k}x_{0,k}}$, $j\geq 1$.
%\end{remark}

%In comparison with Proposition \ref{DefEquiv}, the equivalence relation $\shortstack{$_{{\fontsize{6}{7}\selectfont *}}$\\$\sim$}$ on $\mathcal{S}_{\Lambda}$ is introduced  in a more compact and
%coordinate-free way in Definition \ref{DefEquiv00}. In any case,
We will use $\shortstack{$_{{\fontsize{6}{7}\selectfont *}}$\\$\sim$}$ by restriction to the case of the exponential sums in $\mathcal{S}_{\Lambda}$ of a real or complex variable, and analogously for trigonometric polynomials. It is obvious that $\shortstack{$_{{\fontsize{6}{7}\selectfont *}}$\\$\sim$}$ is independent of the basis $G_{\Lambda}$ for $\Lambda$. %, and the coefficients of equivalent exponential sums have the same modulus.

Finally, we next show that the translates $A(p+i\tau)$, $\tau\in\mathbb{R}$, of an exponential sum $A(p)$ in $\mathcal{S}_{\Lambda}$ are invariant with respect to the equivalence class generated by the relation $\shortstack{$_{{\fontsize{6}{7}\selectfont *}}$\\$\sim$}$. %of Definition \ref{DefEquiv}.

\begin{lemma}\label{lem}
Given $\Lambda$ a set of exponents, let $A(p)\in\mathcal{S}_{\Lambda}$. Then the exponential sums included in
$\mathcal{T}_A=\{A_{\tau}(p):=A(p+i\tau):\tau\in\mathbb{R}\}$ are in the same equivalence class of $\mathcal{S}_{\Lambda}/\shortstack{$_{{\fontsize{6}{7}\selectfont *}}$\\$\sim$}$. %$\mathcal{T}_f=\{f_{\tau}(s)=f(s+i\tau):\tau\in\mathbb{R}\}$
\end{lemma}
\begin{proof}
Let $A(p)=\sum_{j\geq 1}a_je^{\lambda_jp}\ \mbox{with }a_j\in\mathbb{C},\ \lambda_j\in\Lambda$. Then for all real number $\tau$ the sum $A_{\tau}(p):=A(p+i\tau)$, defined formally as $\sum_{j\geq 1}a_je^{i\tau\lambda_j}e^{\lambda_jp}$, can be written as
$A_{\tau}(p)=\sum_{j\geq 1}b_je^{\lambda_jp},$
with $b_j:=a_je^{i\tau\lambda_j}=a_je^{i\tau<\mathbf{r}_j,\mathbf{g}>}=a_je^{i<\mathbf{r}_j,\tau\mathbf{g}>}$, where the vectors $\mathbf{r}_j$ and $\mathbf{g}$ are defined above. Hence, taking $\mathbf{x}_n=\tau\mathbf{g}$ for any integer value $n\geq 1$, Proposition \ref{DefEquiv} holds, i.e. $A\ \shortstack{$_{{\fontsize{6}{7}\selectfont *}}$\\$\sim$}\  A_{\tau}$ for all real $\tau$.
\end{proof}

\section{The finite exponential sums of the classes $\mathcal{P}_{\Lambda}$}\label{section3}

In this section we are going to consider the following classes of finite exponential sums, which can also be viewed as subclasses of those $\mathcal{S}_{\Lambda}$ of the previous section.

\begin{definition}
Let $\Lambda=\{\lambda_1,\ldots,\lambda_n\}$ be a set of $n\geq 1$ distinct real numbers, which we will call a set of exponents or frequencies. We will say that a function $f:\mathbb{C}\mapsto\mathbb{C}$ (resp. $f:\mathbb{R}\mapsto\mathbb{C}$) is in the class $\mathcal{P}_{\Lambda}$ (resp. $\mathcal{P}_{\mathbb{R},\Lambda}$) if it is a finite exponential sum of the form
 \begin{equation}\label{eqq0new}
f(s)=a_1e^{\lambda_1s}+\ldots+a_ne^{\lambda_ns},\ a_j\in\mathbb{C},\ \lambda_j\in\Lambda,\ j=1,\ldots,n.
\end{equation}
\begin{equation}\label{eqq0new2}
\mbox{(Resp. }f(t)=a_1e^{i\lambda_1t}+\ldots+a_ne^{i\lambda_nt},\ a_j\in\mathbb{C},\ \lambda_j\in\Lambda,\ j=1,\ldots,n.\mbox{)}
\end{equation}
\end{definition}

The functions $f(s)$, $s=\sigma+it$, of type (\ref{eqq0new}) are also called exponential polynomials or Dirichlet polynomials and they are entire functions (see also for example \cite{equivalenceclasses}). Besides, the functions $f(t)$ of type (\ref{eqq0new2}) are also called trigonometric polynomials and they are also obtained from functions of type (\ref{eqq0new}) by just restricting their domain to a vertical line $\sigma=\sigma_0,$ with $\sigma_0\in\mathbb{R}$.

The equivalence relation which was introduced in Definition \ref{DefEquiv00} can naturally be applied to the exponential sums in $\mathcal{P}_{\Lambda}$ and $\mathcal{P}_{\mathbb{R},\Lambda}$.
In the context of equivalent finite exponential sums, we next prove an important theorem that will later be used in order to prove one of our main results in this paper.

\begin{theorem}\label{prop3}
Given $\Lambda=\{\lambda_1,\lambda_2,\ldots,\lambda_n\}$ a finite set of exponents, let
$f_1(s)=\sum_{j=1}^na_je^{\lambda_j s}\ \mbox{and }f_2(s)=\sum_{j=1}^nb_je^{\lambda_j s}$
be two equivalent functions in the class $\mathcal{P}_{\Lambda}$. Fixed $\sigma_0,\sigma_1\in\mathbb{R}$, with $\sigma_0<\sigma_1$ and $\varepsilon>0$, there exists a relatively dense set of real numbers $\tau$ such that
$$|f_1(s+i\tau)-f_2(s)|<\varepsilon\ \ \forall s\in\{\sigma+it\in\mathbb{C}:\sigma_0\leq\sigma\leq\sigma_1\}.$$
\end{theorem}
\begin{proof}
Let $G_{\Lambda}=\{g_1,\ldots, g_m\}$, for a certain $m\geq1$, be linearly independent over the rationals so that each $\lambda_j\in\Lambda$ can be expressible as a linear combination of its terms, say
\begin{equation}\label{u}
\lambda_j=\sum_{k=1}^{m}r_{j,k}g_k,\ \mbox{for some }r_{j,k}=\frac{p_{j,k}}{q_{j,k}}\in\mathbb{Q},\ j=1,2,\ldots,n.
\end{equation}
Consider $\varepsilon>0$, $q:=\operatorname{lcm}(q_{j,k}: j=1,\ldots,n,k=1,\ldots,m)$, $r:=\max\{|r_{j,k}|: j=1,\ldots,n,k=1,\ldots,m\}>0$, $a:=\max\{|a_j|:j=1,2,\ldots,n\}>0$ and $E:=\max\{e^{\lambda_j\sigma_0},e^{\lambda_j\sigma_1}:1\leq j\leq n\}$. Since $f_1\ \shortstack{$_{{\fontsize{6}{7}\selectfont *}}$\\$\sim$}\ f_2$, there exists a vector of real numbers $\mathbf{x}_0=(x_{0,1},x_{0,2},\ldots,x_{0,m})$ such that
\begin{equation}\label{un}
b_j=a_j e^{<\mathbf{r}_j,\mathbf{x}_0>i}=a_je^{i\sum_{k=1}^{m}r_{j,k}x_{0,k}},\ j=1,2,\ldots,n.
\end{equation}
Now, as the numbers $
c_{k}=\frac{g_k}{2\pi q},\text{ }k=1,2,\ldots,m,
$ are rationally independent,
we next apply Kronecker's theorem
\cite[p.382]{Hardy} with the following choice:
$c_k$, $\varepsilon_1=\frac{ 1}{2\pi q }\cdot\frac{\varepsilon/2}{a\cdot m\cdot n\cdot r\cdot E}>0$ and
$
d_{k}=\frac{x_{0,k}}{2\pi q},\text{
}k=1,2,\ldots,m.
$
In this manner we assure the existence of a real number $\tau_1>d>0$ and integer numbers
$e_{1},e_{2},\ldots,e_{m}$ such that
\[
\left\vert \tau_1 c_k-e_k-d_k\right\vert=\left|\frac{\tau_1 g_k}{2\pi q}-e_k-\frac{x_{0,k}}{2\pi q}\right| <\varepsilon_1,
\]
that is
\begin{equation}\label{un2}
\tau_1 g_k=2\pi qe_k+x_{0,k}+\eta_k, \mbox{with  }|\eta_k|<2\pi q\varepsilon_1.
\end{equation}
Therefore, from (\ref{u}) and (\ref{un}), with $\sigma_0\leq\sigma\leq\sigma_1$ and $t\in \mathbb{R}$, we have
$$|f_1(\sigma+it+i\tau_1)-f_2(\sigma+it)|=\left|\sum_{j=1}^na_je^{\lambda_j (\sigma+it)}e^{i\tau_1\lambda_j}-\sum_{j=1}^n a_je^{\lambda_j (\sigma+it)}e^{i\sum_{k=1}^{m}r_{j,k}x_{0,k}}\right|\leq$$
$$\sum_{j=1}^n|a_j|e^{\lambda_j\sigma}\left|e^{i\tau_1\lambda_j}-e^{i\sum_{k=1}^{m}r_{j,k}x_{0,k}}\right|\leq a\sum_{j=1}^n e^{\lambda_j\sigma}\left|e^{i\tau_1\lambda_j}-e^{i\sum_{k=1}^{m}r_{j,k}x_{0,k}}\right|=$$
$$a\sum_{\substack{1\leq j\leq n\\\lambda_j<0}}e^{\lambda_j\sigma}\left|e^{i\tau_1\lambda_j}-e^{i\sum_{k=1}^{m}r_{j,k}x_{0,k}}\right|+a\sum_{\substack{1\leq j\leq n\\\lambda_j\geq 0}}e^{\lambda_j\sigma}\left|e^{i\tau_1\lambda_j}-e^{i\sum_{k=1}^{m}r_{j,k}x_{0,k}}\right|\leq$$
$$a\left(\sum_{\substack{1\leq j\leq n\\\lambda_j<0}}e^{\lambda_j\sigma_0}\left|e^{i\tau_1\lambda_j}-e^{i\sum_{k=1}^{m}r_{j,k}x_{0,k}}\right|+\sum_{\substack{1\leq j\leq n\\\lambda_j\geq 0}}e^{\lambda_j\sigma_1}\left|e^{i\tau_1\lambda_j}-e^{i\sum_{k=1}^{m}r_{j,k}x_{0,k}}\right|\right)\leq$$
$$a E\sum_{j=1}^n\left|e^{i\tau_1\lambda_j}-e^{i\sum_{k=1}^{m}r_{j,k}x_{0,k}}\right|=a E\sum_{j=1}^n\left|e^{i\tau_1\sum_{k=1}^{m} r_{j,k}g_k}-e^{i\sum_{k=1}^{m}r_{j,k}x_{0,k}}\right|,$$
which, from (\ref{un2}), is equal to
$$a E\sum_{j=1}^n\left|e^{i\sum_{k=1}^{m} (r_{j,k}2\pi qe_k+r_{j,k}x_{0,k}+r_{j,k}\eta_k)}-e^{i\sum_{k=1}^{m}r_{j,k}x_{0,k}}\right|=$$
$$a E\sum_{j=1}^n\left|e^{i\sum_{k=1}^{m}r_{j,k}\eta_k}-1\right|\leq a E\sum_{j=1}^n\left|\sum_{k=1}^{m}r_{j,k}\eta_k\right|\leq$$
$$anr E\sum_{k=1}^{m}\left|\eta_k\right|<anr E\sum_{k=1}^{m}\frac{\varepsilon/2}{a\cdot m\cdot n\cdot r\cdot E}=\varepsilon/2.$$
That is
\begin{equation}\label{jhg}
|f_1(\sigma+it+i\tau_1)-f_2(\sigma+it)|< \varepsilon/2,\ \mbox{with }\sigma_0\leq\sigma\leq\sigma_1 \mbox{ and }t\in \mathbb{R}.
\end{equation}
Moreover, since $f_1(s)$ is an
almost-periodic function, there exists a real number $l=l(\varepsilon)$%
\ such that every interval of length $l$ on the imaginary axis
contains at
least one translation number $i\tau$, associated with $\varepsilon$, satisfying
\begin{equation}\label{bgt}
\left\vert f_1(\sigma+it+i\tau)-f_1(\sigma+it)\right\vert \leq \varepsilon/2\ \mbox{for all }\sigma+it \mbox{ on a given reduced strip}.
 \end{equation}
 %with $\sigma_0\leq\sigma\leq\sigma_1$ and $t\in \mathbb{R}$.
Consequently, from (\ref{jhg}) and (\ref{bgt}) we deduce the existence of a relatively dense set of real numbers $\tau$ such that any $s\in\{\sigma+it\in\mathbb{C}:\sigma_0\leq\sigma\leq\sigma_1\}$ satisfies
$$|f_1(s+i(\tau+\tau_1))-f_2(s)|\leq |f_1(s+i(\tau+\tau_1))-f_1(s+i\tau_1)|+|f_1(s+i\tau_1)-f_2(s)|<\varepsilon.$$
This proves the result.
\end{proof}

As an immediate consequence of Theorem \ref{prop3}, we obtain the following corollary.

\begin{corollary}\label{corol3}
Given $\Lambda=\{\lambda_1,\lambda_2,\ldots,\lambda_n\}$ a finite set of exponents, let
$f_1(t)=\sum_{j=1}^na_je^{i\lambda_j t}$ and $f_2(t)=\sum_{j=1}^nb_je^{i\lambda_j t}$
be two equivalent functions in the class $\mathcal{P}_{\mathbb{R},\Lambda}$. Fixed $\varepsilon>0$, there exists a relatively dense set of real numbers $\tau$ such that
$$|f_1(t+\tau)-f_2(t)|<\varepsilon\ \ \forall t\in\mathbb{R}.$$
\end{corollary}
%\begin{proof}
%The result directly follows from Theorem \ref{prop3}.
%\end{proof}

In the same manner that Bohr's notion of almost periodicity is equivalent to that of Bochner, the fact that the set of real numbers $\tau$ satisfying the results above is relatively dense seems to indicate that a similar property is satisfied in our context. In this sense, with respect to the topology of uniform convergence, we include the following important result on the equivalence classes in either $\mathcal{P}_{\mathbb{R},\Lambda}/\shortstack{$_{{\fontsize{6}{7}\selectfont *}}$\\$\sim$}$ or $\mathcal{P}_{\Lambda}/\shortstack{$_{{\fontsize{6}{7}\selectfont *}}$\\$\sim$}$.
\begin{proposition}\label{proppp}
Let $\Lambda$ be a finite set of exponents and $\mathcal{G}$ an equivalence class in either $\mathcal{P}_{\mathbb{R},\Lambda}/\shortstack{$_{{\fontsize{6}{7}\selectfont *}}$\\$\sim$}$ or $\mathcal{P}_{\Lambda}/\shortstack{$_{{\fontsize{6}{7}\selectfont *}}$\\$\sim$}$. Then $\mathcal{G}$ is sequentially compact.
\end{proposition}
\begin{proof}
%Note first that, by using Lemma \ref{defnueva2}, if $f\in AP(U,\mathbb{C})$ then any function of its equivalence class is also included in $AP(U,\mathbb{C})$. %$\mathcal{G}$ is almost periodic on the same open vertical strip $U$.
Let $\Lambda=\{\lambda_1,\ldots,\lambda_n\}$ be a set of $n\geq 1$ distinct exponents and $\{P_l(t)\}_{l\geq 1}$, with $P_l(t)=a_{l,1}e^{i\lambda_1t}+\ldots+a_{l,n}e^{i\lambda_nt}$, a sequence in an equivalence class $\mathcal{G}$ of $\mathcal{P}_{\mathbb{R},\Lambda}/\shortstack{$_{{\fontsize{6}{7}\selectfont *}}$\\$\sim$}$ (the case $\mathcal{P}_{\Lambda}/\shortstack{$_{{\fontsize{6}{7}\selectfont *}}$\\$\sim$}$ is analogous). %Then $|a_{1,j}|=|a_{2,j}|=\ldots=|a_{l,j}|=\ldots$ for every $j=1,\ldots,n$.
Fixed a basis $G_{\Lambda}=\{g_1,g_2,\ldots,g_m\}$ for $\Lambda$, let $\mathbf{r}_j=(r_{j,1},r_{j,2},\ldots,r_{j,m})$ be the vector verifying $<\mathbf{r}_j,\mathbf{g}>=\lambda_j$ for each $j=1,\ldots,n$, where $\mathbf{g}=(g_1,g_2,\ldots,g_m)$. %is the vector of the basis for $\Lambda$.
Since $f_1\ \shortstack{$_{{\fontsize{6}{7}\selectfont *}}$\\$\sim$}\ f_l$ for each $l=1,2,\ldots$, we deduce from Proposition \ref{DefEquiv} that there exists  $\mathbf{x}_l=(x_{l,1},x_{l,2},\ldots,x_{l,m})\in\mathbb{R}^{m}$ such that
\begin{equation*}\label{seqaddpi}
a_{l,j}=a_{1,j}e^{i<\mathbf{r}_j,\mathbf{x}_l>}=a_{1,j}\prod_{k=1}^{m}e^{ir_{j,k}x_{l,k}},\ j=1,2\ldots,n.
\end{equation*}
%That is, for each $l=1,2,\ldots$, the Fourier series of the functions $f_l(t)$ can be written as
%\begin{equation*}\label{seqaddpi}
%\sum_{j\geq 1}\left(a_{1,j}\prod_{k=1}^{m}e^{ir_{j,k}x_{l,k}}\right)e^{i\lambda_jt}.
%\end{equation*}
%It is immediate that $\{f_l\}_{l\geq 1}$ is uniformly bounded on each compact subset of $U$. Thus Montel theorem \cite{Titchmarsh2} assures the existence of a subsequence $\{f_{l_m}\}\subset \{f_l\}$ which converges uniformly on compact subsets of $U$. Let $f$ denote the limit of $\{f_{l_m}\}$.
Let $r_{j,k}=\frac{p_{j,k}}{q_{j,k}}$ with $p_{j,k}$ and $q_{j,k}$ coprime integer numbers, and define $q_k:=\operatorname{lcm}(q_{1,k},q_{2,k},\ldots,q_{m,k})$ for each $k=1,\ldots,m$. Then
$e^{ir_{j,k}(x_{l,k}+2\pi q_k)}=e^{ir_{j,k}x_{l,k}}$ for each $j=1,\ldots,n$ and $k=1,\ldots,m$, which means that, without loss of generality, we can take $\mathbf{x}_l=(x_{l,1},x_{l,2},\ldots,x_{l,m})\in [0,2\pi q_1)\times[0,2\pi q_2)\times\cdots\times[0,2\pi q_m)$ for any $l\geq 1$.
Hence we can suppose that the sequence $\{x_{l,1}\}_{l\geq 1}$ is bounded, which implies that there exists a subsequence $\{x_{l_{m,1},1}\}_{m\geq 1}\subset \{x_{l,1}\}_{l\geq 1}$ convergent to a point $x_{0,1}$. From the sequence $\{x_{l_{m,1},2}\}_{m\geq 1}$, which is also bounded, we can draw a subsequence $\{x_{l_{m,2},2}\}_{m\geq 1}$ convergent to a point $x_{0,2}$, and so for each $j=1,2,\ldots,n$ we get a vector $\mathbf{x}_0=(x_{0,1},x_{0,2},\ldots,x_{0,m})\in [0,2\pi q_1)\times[0,2\pi q_2)\times\cdots\times[0,2\pi q_m)$ such that the sequence of coefficients $\{a_{l,j}\}_{l\geq 1}$ contains a subsequence converging to $a_j:=a_{1,j}\prod_{k= 1}^m e^{ir_{j,k}x_{0,k}}=a_{1,j}e^{i<\mathbf{r}_j,\mathbf{x}_0>}$ for each $j=1,2\ldots,n$. Consequently, $P(t):=a_{1}e^{i\lambda_1t}+\ldots+a_{n}e^{i\lambda_nt}$ is in $\mathcal{G}$ and the result holds.
\end{proof}

\begin{remark}\label{add}
With respect to the proposition above, it is worth noting that we can assure the existence of infinitely many vectors $\mathbf{x}_0=(x_{0,1},x_{0,2},\ldots,x_{0,m})\in\mathbb{R}^m$ such that a convergent subsequence of $\{P_l(t)\}_{l\geq 1}\subset \mathcal{G}$ can be extracted so that
$a_j=a_{1,j}e^{i<\mathbf{r}_j,\mathbf{x}_0>}$ for each $j=1,2\ldots,n$, where the $a_j$'s are the Fourier coefficients of the limit function and the $a_{1,j}$'s are the Fourier coefficients of $P_1(t)$. Indeed, the vectors $\mathbf{x}_0$ can be constructed by fixing $\mathbf{x}_l$ in $ [2n_1\pi q_1,2(n_1+1)\pi q_1)\times[2n_2\pi q_2,2(n_2+1)\pi q_2)\times\cdots\times[2n_m\pi q_m,2(n_m+1)\pi q_m)$ for any $l\geq 1$, where $n_k\in\mathbb{Z}$.

In this same respect, suppose that $P_1(t)=a_1e^{i\lambda_1t}+\ldots+a_ke^{i\lambda_kt}$ and $Q_1(t)=b_1e^{i\lambda_1t}+\ldots+b_ke^{i\lambda_kt}$ are two equivalent trigonometric polynomials in $\mathcal{P}_{\mathbb{R},\Lambda_1}$, and let $m_1$ be the number of elements of any basis for $\Lambda_1=\{\lambda_1,\ldots,\lambda_k\}$. Thus there exist infinitely many vectors of the form $\mathbf{x}_0^{(1)}=(x_{0,1},x_{0,2},\ldots,x_{0,m_1})\in\mathbb{R}^{m_1}$ such that
$b_j=a_{j}e^{i<\mathbf{r}_j,\mathbf{x}_0^{(1)}>}$ for each $j=1,\ldots,k$. Now, if $P_2(t)=a_1e^{i\lambda_1t}+\ldots+a_ke^{i\lambda_kt}+a_{k+1}e^{i\lambda_{k+1}t}+\ldots+a_ne^{i\lambda_nt}$ is equivalent to $Q_2(t)=b_1e^{i\lambda_1t}+\ldots+b_ke^{i\lambda_kt}+b_{k+1}e^{i\lambda_{k+1}t}+\ldots+b_ne^{i\lambda_nt}$ and we denote as $m_2$ the number of elements of any basis for $\Lambda_2=\{\lambda_1,\ldots,\lambda_k,\lambda_{k+1},\ldots,\lambda_n\}$, it is clear that there exist infinitely many vectors $\mathbf{x}_0^{(2)}\in\mathbb{R}^{m_2}$ such that $b_j=a_{j}e^{i<\mathbf{r}_j,\mathbf{x}_0^{(2)}>}$ for each $j=1,\ldots,n$, and the first $m_1$ components of $\mathbf{x}_0^{(2)}$ are subsets of the vectors $\mathbf{x}_0^{(1)}$.
\end{remark}

\section{Main results}

In general terms, the equivalence relation of Definition \ref{DefEquiv00} can be immediately adapted to the case of the Besicovitch space $B(\mathbb{R},\mathbb{C})$ which is obtained by the completion of the trigonometric polynomials with respect to the seminorm given by
$\limsup_{l\to\infty}\left(\frac{ 1}{ 2l}\int_{-l}^{l}|f(t)|dt\right)$ (with respect to the properties of this space see for example \cite[Section 3.4]{Corduneanu}). In particular, the space of functions $B(\mathbb{R},\mathbb{C})$ contains those of the space of the almost periodic functions $AP(\mathbb{R},\mathbb{C})$ and every function in $B(\mathbb{R},\mathbb{C})$ is associated with a real exponential sum with real frequencies of the form $\sum_{j\geq 1}a_je^{i\lambda_jt}$ (which will also be called its Fourier series). In general, when we write that a function $f$ is in these spaces we do not have in mind the function $f$ itself, it does represent a whole class of equivalent functions according to the relation $f_1\simeq f_2$ if and only if $\limsup_{l\to\infty}\left(\frac{ 1}{ 2l}\int_{-l}^{l}|f(t)-g(t)|dt\right)=0.$
This set of complex functions with real variable could be generalized to the case of complex functions
%This development can be naturally extended to the case of complex functions
$f:U\to\mathbb{C}$, with $U=\{\sigma+it\in\mathbb{C}:\alpha<\sigma<\beta\}$ a vertical strip in $\mathbb{C}$, which lead us to the functional sets $B(U,\mathbb{C})\supset AP(U,\mathbb{C})$ whose functions $f(s)$, $s=\sigma+it$, are associated with a complex exponential sum with real frequencies of the form $\sum_{j\geq 1}a_je^{\lambda_js}$ (which will also be called its Dirichlet series) which represents the Fourier series of $f(s)$ on any vertical line in $U$, with $f(\sigma_0+it):\mathbb{R}\to\mathbb{C}$ in $B(\mathbb{R},\mathbb{C})$ for every $\sigma_0\in(\alpha,\beta)$.
As before, when we write that a function $f$ is in the spaces $B(U,\mathbb{C})$ we do not have in mind the function $f$ itself, it does also represent a whole class of equivalent functions according to the relation above $\simeq$ on every vertical line in $U$ (by abuse of notation, it will also be denoted as $\simeq$ in the following definition).

In this way, we establish the following definition.

\begin{definition}[mod.]\label{DefEquiv2}
Given $\Lambda=\{\lambda_1,\lambda_2,\ldots,\lambda_j,\ldots\}$ a set of exponents, let $f_1$ and $f_2$ denote two  equivalence classes of $B(\mathbb{R},\mathbb{C})/\simeq$ (resp. $B(U,\mathbb{C})/\simeq$) whose associated Fourier series (resp. Dirichlet series) are given by
 \begin{equation*}\label{eqq}
\sum_{j\geq 1}a_je^{i\lambda_jt}\ \mbox{and}\ \sum_{j\geq 1}b_je^{i\lambda_jt},\ a_j,b_j\in\mathbb{C},\ \lambda_j\in\Lambda.\mbox{)}
\end{equation*}
 \begin{equation*}\label{eqq00}
\mbox{(resp. }\sum_{j\geq 1}a_je^{\lambda_js}\ \mbox{and}\ \sum_{j\geq 1}b_je^{\lambda_js},\ a_j,b_j\in\mathbb{C},\ \lambda_j\in\Lambda.
\end{equation*}
We will say that $f_1$ is equivalent to $f_2$ if for each integer value $n\geq1$, with $n\leq\sharp\Lambda$, it is satisfied $a_n^*\sim b_n^*$, where $a_n^*,b_n^*:\{\lambda_1,\ldots,\lambda_n\}\to\mathbb{C}$ are the functions given by $a_n^*(\lambda_j):=a_j$ and $b_n^*(\lambda_j):=b_j$, $j=1,2,\ldots,n$, and $\sim$ is in Definition \ref{DefEquiv0}.
%We will say that $f_1$ is equivalent to $f_2$ if $a\sim b$, where $a,b:\Lambda\to\mathbb{C}$ are the functions given by $a(\lambda_j):=a_j$ and $b(\lambda_j):=b_j$, $j=1,2,\ldots$, and $\sim$ is in Definition \ref{DefEquiv0}.
%
%
%Fixed a basis for $\Lambda$, for each $j\geq1$ let $\mathbf{r}_j$ be the vector of rational components verifying (\ref{errej}).
%%-%For each $j\geq1$, let $\mathbf{r}_j$ be the vector of rational components that verifies  $$\lambda_j=<\mathbf{r}_j,\mathbf{g}>=\sum_{k=1}^{q_j}r_{j,k}g_k,$$ where
%%  $\mathbf{g}=(g_1,g_2,\ldots,g_k,\ldots)$ is the vector of the elements of a basis $G_{\Lambda}$ for $\Lambda$.
%We will say that $f_1$ is equivalent to $f_2$, relative to the basis $G_{\Lambda}$, if there exists $\mathbf{x}_0=(x_{0,1},x_{0,2},\ldots,x_{0,k},\ldots)\in \mathbb{R}^{\sharp G_{\Lambda}}$ such that $b_j=a_j e^{<\mathbf{r}_j,\mathbf{x}_0>i}$ for $j\geq 1$. In that case, we will write $f_1\sim f_2$.
\end{definition}
Again, by abuse of notation, we will also use $\shortstack{$_{{\fontsize{6}{7}\selectfont *}}$\\$\sim$}$ for the equivalence relation above, which can also be characterized in terms of a basis for $\Lambda$ (as in Proposition \ref{DefEquiv}). %introduced in definitions \ref{DefEquiv0} and \ref{DefEquiv00}.

For our purposes, we next focus our attention on the following classes of functions included in the spaces of almost periodic functions $AP(\mathbb{R},\mathbb{C})$ and $AP(U,\mathbb{C})$. We recall that the elements of these spaces also correspond to the so-called uniformly almost periodic functions which were used in \cite{Besi}. %and also handled in \cite{Corduneanu}. %<--añadir más

\begin{definition}\label{DF}
Let $\Lambda=\{\lambda_1,\lambda_2,\ldots,\lambda_j,\ldots\}$ be an arbitrary countable set of distinct real numbers. We will say that a function $f:U\subset\mathbb{C}\to\mathbb{C}$ (resp. $f:\mathbb{R}\to\mathbb{C}$) is in the class $\mathcal{D}_{\Lambda}$ (resp. $\mathcal{F}_{\Lambda}$) if it is an almost periodic function in $AP(U,\mathbb{C})$ (resp. in $AP(\mathbb{R},\mathbb{C})$) whose associated Dirichlet series (resp. Fourier series) is of the form
 \begin{equation}\label{eqqo}
\sum_{j\geq 1}a_je^{\lambda_js},\ a_j\in\mathbb{C},\ \lambda_j\in\Lambda,
\end{equation}
 \begin{equation}\label{eqq00o}
\mbox{(resp. }\sum_{j\geq 1}a_je^{i\lambda_jt},\ a_j\in\mathbb{C},\ \lambda_j\in\Lambda.\mbox{)}
\end{equation}
where $U$ is a strip of the type $\{s\in\mathbb{C}: \alpha<\operatorname{Re}s<\beta\}$, with $-\infty\leq\alpha<\beta\leq\infty$.
\end{definition}

It is convenient to recall that any almost periodic function in $AP(U,\mathbb{C})$ (resp. in $AP(\mathbb{R},\mathbb{C})$) is uniquely determined by its Dirichlet series (resp. Fourier series), which is of type (\ref{eqqo}) (resp. of type (\ref{eqq00o})). In fact, even in the case that the sequence of the partial sums of its Dirichlet series (resp. Fourier series) does not converge uniformly, there exists a sequence of finite exponential sums, called Bochner-Fej\'{e}r polynomials, of the type $P_k(s)=\sum_{j\geq 1}p_{j,k}a_je^{\lambda_js}$ (resp. $P_k(t)=\sum_{j\geq 1}p_{j,k}a_je^{i\lambda_jt}$) where for each $k$ only a finite number of the factors $p_{j,k}$ differ from zero, which converges uniformly to $f$ in every reduced strip in $U$ (resp. in $\mathbb{R}$) and converges formally to the Dirichlet series on $U$ (resp. to the Fourier series on $\mathbb{R}$) \cite[Polynomial approximation theorem, pgs. 50,148]{Besi}.

In this way, Definition \ref{DefEquiv2} can be particularized to the classes $\mathcal{D}_{\Lambda}$ and $\mathcal{F}_{\Lambda}$ for which the equivalence classes of $\mathcal{D}_{\Lambda}/\simeq$ and $\mathcal{F}_{\Lambda}/\simeq$ are reduced to individual functions.

We next demonstrate the following important lemma which will let us to consider equivalence classes in the space of almost periodic functions.

\begin{lemma}\label{defnueva}
Let $f_1(t)\in AP(\mathbb{R},\mathbb{C})$ be an almost periodic function whose Fourier series is given by $\sum_{j\geq 1}a_je^{i\lambda_jt},\ a_j\in\mathbb{C}$, where $\{\lambda_1,\ldots,\lambda_j,\ldots\}$ is a set of distinct exponents. Consider $b_j\in\mathbb{C}$ such that $\sum_{j\geq 1}a_je^{i\lambda_jt}$ is equivalent to $\sum_{j\geq 1}b_je^{i\lambda_jt}$. Then $\sum_{j\geq 1}b_je^{i\lambda_jt}$ is the Fourier series associated with an almost periodic function $f_2(t)\in AP(\mathbb{R},\mathbb{C})$ such that $f_1\ \shortstack{$_{{\fontsize{6}{7}\selectfont *}}$\\$\sim$}\ f_2$.
\end{lemma}
\begin{proof}
Take $\Lambda=\{\lambda_1,\ldots,\lambda_j,\ldots\}$. By the hypothesis, $f_1\in\mathcal{F}_{\Lambda}\subset AP(\mathbb{R},\mathbb{C})$ is determined by the series
$\sum_{j\geq 1}a_je^{i\lambda_jt},\ a_j\in\mathbb{C},\ \lambda_j\in\Lambda.$ Moreover,
since it is accomplished $\sum_{j\geq 1}a_je^{i\lambda_jt}\shortstack{$_{{\fontsize{6}{7}\selectfont *}}$\\$\sim$} \sum_{j\geq 1}b_je^{i\lambda_jt}$, for each integer value $n\geq 1$ there exists $\mathbf{x}_n\in\mathbb{R}^{\sharp\Lambda}$ such that $b_j=a_j e^{<\mathbf{r}_j,\mathbf{x}_n>i}$ for each $j=1,\ldots,n$, where $\mathbf{r}_j$ is given by (\ref{errej}).
On the other hand, let $$P_k(t)=\sum_{j\geq 1}p_{j,k}a_je^{i\lambda_jt},\ k=1,2,\ldots,$$ with $p_{j,k}\to 1$ as $k\to\infty$, be the sequence of Bochner-Fej\'{e}r polynomials which converges uniformly on $\mathbb{R}$ to $f_1(t)$ (\cite[Polynomial approximation theorem, p. 50]{Besi}). It is known that $\{P_k(t)\}_{k\geq 1}$ is equicontinuous on $\mathbb{R}$, equi-almost periodic on $\mathbb{R}$  and it converges in mean (see \cite[Section 4.4]{Corduneanu}). %pg 143-144
In particular, that implies that, fixed $\varepsilon>0$, there exists $\delta>0$ such that for any $k\geq1$ we have
\begin{equation}\label{Ju}
|P_k(t_1)-P_k(t_2)|<\varepsilon/3\ \mbox{if }|t_1-t_2|<\delta.
\end{equation}
Also, there exists $l>0$ such that any interval $(a,a+l)$ of length $l$ contains a number $\tau$ satisfying
\begin{equation}\label{Lo}
|P_k(t+\tau)-P_k(t)|<\varepsilon/3\ \mbox{for any }t\in\mathbb{R}\mbox{ and }k\geq 1.
\end{equation}
Now, define the sequence of trigonometric polynomials
$$Q_k(t):=\sum_{j\geq 1}p_{j,k}b_je^{i\lambda_jt},\ k=1,2,\ldots,$$
where, fixed $k$, we can take $b_j=a_je^{<\mathbf{r}_j,\mathbf{x}_{n_k}>i}$ with $n_k$ the greatest integer value $j$ such that $p_{j,k}\neq 0$.
%$$Q_k(t):=\sum_{j\geq 1}p_{j,k}a_je^{<\mathbf{r}_j,\mathbf{x}_0>i}e^{i\lambda_jt},\ k=1,2,\ldots$$
By Corollary \ref{corol3} there exists $\tau_1>0$ such that
\begin{equation}\label{tos}
|Q_k(t+\tau_1)-P_k(t)|<\varepsilon/3\ \mbox{for any }t\in\mathbb{R}.
\end{equation}
Consequently, we deduce from (\ref{Ju}) and (\ref{tos}) that, if $|t_1-t_2|<\delta$, we have
\begin{multline*}
$$|Q_k(t_1+\tau_1)-Q_k(t_2+\tau_1)|\leq|Q_k(t_1+\tau_1)-P_k(t_1)|+\\|P_k(t_1)-P_k(t_2)|+ |P_k(t_2)-Q_k(t_2+\tau_1)|< \varepsilon \mbox{ for each }k\geq 1.$$
\end{multline*}
Hence $\{Q_k(t)\}$ is equicontinuous.
Also, by (\ref{Lo}) and (\ref{tos}), for each $k\geq 1$ we have
\begin{multline*}
$$|Q_k(t+\tau)-Q_k(t)|=|Q_k(t+\tau)-P_k(t+\tau-\tau_1)|+\\|P_k(t+\tau-\tau_1)-P_k(t-\tau_1)|+|P_k(t-\tau_1)-Q_k(t)|<\varepsilon.$$
\end{multline*}
Hence $\{Q_k(t)\}$ is equi-almost periodic on $\mathbb{R}$.
Finally, since $\{P_k(t)\}_{k\geq 1}$ converges in mean, there exists $k_0$ such that
\begin{equation*}\label{Juan}
M\{|P_{k_1}(t)-P_{k_2}(t)|^2\}<\varepsilon\ \mbox{for all }k_1,k_2\geq k_0,
\end{equation*}
where $M\{|P_{k_1}(t)-P_{k_2}(t)|^2\}=\Sum_{j\geq 1}(p_{j,k_1}-p_{j,k_2})^2|a_j|^2$.
However, note that
$$M\{|Q_{k_1}(t)-Q_{k_2}(t)|^2\}=\Sum_{j\geq 1}(p_{j,k_1}-p_{j,k_2})^2\left|e^{<\mathbf{r}_j,\mathbf{x}_{n_{k_2}}>i}\right|^2|a_j|^2=$$
$$M\{|P_{k_1}(t)-P_{k_2}(t)|^2\},$$
where we suppose that $n_{k_2}\geq n_{k_1}$,
which implies that
$\{Q_k(t)\}_{k\geq 1}$ converges in mean.
Consequently, in virtue of \cite[p. 43]{Besi} %p. 138, Theorem 3.4
(or \cite[Section 4.4]{Corduneanu}), we have that $\{Q_k(t)\}_{k\geq 1}$
converges uniformly on $\mathbb{R}$ to $f_2(t)$, say. Now,
since $AP(\mathbb{R},\mathbb{C})$ is the closure of the trigonometric polynomials in the sense of uniform convergence,
we have $f_2(t)\in AP(\mathbb{R},\mathbb{C})$ and, by \cite[p. 21]{Besi}, $\sum_{j\geq 1}b_je^{i\lambda_jt}$ represents its Fourier series. Finally, by taking into account Definition \ref{DefEquiv2} (in terms of Proposition \ref{DefEquiv}) we have $f_1\ \shortstack{$_{{\fontsize{6}{7}\selectfont *}}$\\$\sim$}\ f_2$.
\end{proof}

In the same manner, we have an analogous result for the complex case.
\begin{lemma}\label{defnueva2}
Let $f_1(s)\in AP(U,\mathbb{C})$ be an almost periodic function in a vertical strip $U$ whose Dirichlet series is given by $\sum_{j\geq 1}a_je^{\lambda_js},\ a_j\in\mathbb{C}$, where $\{\lambda_1,\ldots,\lambda_j,\ldots\}$ is a set of distinct exponents. Take  $b_j\in\mathbb{C}$ such that $\sum_{j\geq 1}a_je^{\lambda_js}$ is equivalent to $\sum_{j\geq 1}b_je^{\lambda_j s}$. Then $\sum_{j\geq 1}b_je^{\lambda_js}$ is the Dirichlet series associated with an almost periodic function $f_2(s)\in AP(U,\mathbb{C})$ such that $f_1\ \shortstack{$_{{\fontsize{6}{7}\selectfont *}}$\\$\sim$}\ f_2$.
\end{lemma}
\begin{proof}
Let $U=\{s\in\mathbb{C}:\alpha<\operatorname{Re}s<\beta\}$. Take $\alpha<\sigma_1<\sigma_2<\beta$ and $U_1:=\{s\in\mathbb{C}:\sigma_1\leq\operatorname{Re}s\leq\sigma_2\}$. By Lemma \ref{defnueva}, note that $\sum_{j\geq 1}b_je^{\lambda_j\sigma}e^{i\lambda_jt}$ represents for $\sigma=\sigma_1$ and $\sigma=\sigma_2$ the Fourier series of two almost periodic functions $f_{2,\sigma_1}(t)$ and $f_{2,\sigma_2}(t)$ which are equivalent to $f_{1,\sigma_1}(t):=f_1(\sigma_1+it)$ and $f_{1,\sigma_2}(t):=f_1(\sigma_2+it)$, respectively. Now, by \cite[p. 149, Theorem]{Besi}, there exists a function $f_2(s)\in AP(U_1,\mathbb{C})$ such that $f_2(\sigma_1+it)=f_{2,\sigma_1}(t)$, $f_2(\sigma_2+it)=f_{2,\sigma_2}(t)$ and $\sum_{j\geq 1}b_je^{\lambda_js}$ is the Dirichlet series associated with it. As $U_1$ is an arbitrary reduced strip in $U$, the result holds. Finally, by taking into account Definition \ref{DefEquiv2}, we have $f_1\ \shortstack{$_{{\fontsize{6}{7}\selectfont *}}$\\$\sim$}\ f_2$.
\end{proof}

Since two almost periodic functions are equal if and only if they have the same Dirichlet or Fourier series (\cite[p. 148]{Besi} and \cite[Section 4.2]{Corduneanu}), we deduce from the results above that if two functions in $B(\mathbb{R},\mathbb{C})$ or $B(U,\mathbb{C})$ are equivalent and we know that one of them is in $AP(\mathbb{R},\mathbb{C})$ or $AP(U,\mathbb{C})$, then both functions are in $AP(\mathbb{R},\mathbb{C})$ or $AP(U,\mathbb{C})$, respectively.

\subsection{On the space $AP(\mathbb{R},\mathbb{C})$}

\ \\[0.2cm]
From this section, the set of functions of the Besicovitch space $B(\mathbb{R},\mathbb{C})\supset AP(\mathbb{R},\mathbb{C})$ will be taken as the set of reference in the sense that each function in $B(\mathbb{R},\mathbb{C})$ is associated with a Fourier series \cite[Section 4.2]{Corduneanu}. We will consider that this space is endowed with the topology of uniform convergence on $\mathbb{R}$. Under this topology, we next show that the equivalence classes of $\mathcal{F}_{\Lambda}/\shortstack{$_{{\fontsize{6}{7}\selectfont *}}$\\$\sim$}$ are closed. In fact, more specifically, they are sequentially compact. In this respect, it is worth noting that, by Lemma \ref{defnueva}, if $f\in \mathcal{F}_{\Lambda}$, then any function of its equivalence class is also included in $\mathcal{F}_{\Lambda}$.

\begin{proposition}\label{prop}
Let $\Lambda$ be a set of exponents and $\mathcal{G}$ an equivalence class in $\mathcal{F}_{\Lambda}/\shortstack{$_{{\fontsize{6}{7}\selectfont *}}$\\$\sim$}$. Then $\mathcal{G}$ is sequentially compact.
\end{proposition}
\begin{proof}
Let $\{f_l\}_{l\geq 1}$ be a sequence in an equivalence class $\mathcal{G}$ in $\mathcal{F}_{\Lambda}/\shortstack{$_{{\fontsize{6}{7}\selectfont *}}$\\$\sim$}$.
For each $l=1,2,\ldots$, suppose that the Fourier series which is associated with the almost periodic function $f_l(t)$ is given by
$$\sum_{j\geq 1}a_{l,j}e^{i\lambda_jt}\ \mbox{with }a_{l,j}\in\mathbb{C}\setminus\{0\},\ \lambda_j\in\Lambda.$$
Fixed a basis $G_{\Lambda}=\{g_1,g_2,\ldots,g_k,\ldots\}$ for $\Lambda$, let $\mathbf{r}_j=(r_{j,1},r_{j,2},\ldots)$ be the vector verifying $<\mathbf{r}_j,\mathbf{g}>=\lambda_j$ for each $j\geq 1$, where $\mathbf{g}=(g_1,g_2,\ldots,g_k,\ldots)$. %is the vector of the basis for $\Lambda$.
Since $f_1\ \shortstack{$_{{\fontsize{6}{7}\selectfont *}}$\\$\sim$}\ f_l$ for each $l=1,2,\ldots$, we deduce from Proposition \ref{DefEquiv} that for each integer value $n\geq 1$ there exists  $\mathbf{x}_{l,n}=(x_{l,n,1},x_{l,n,2},\ldots)\in\mathbb{R}^{\sharp G_\Lambda}$ such that
\begin{equation}\label{seqaddpipipi1}
a_{l,j}=a_{1,j}e^{i<\mathbf{r}_j,\mathbf{x}_{l,n}>},\ j=1,2\ldots,n\ \mbox{with }\lambda_j\in\Lambda.
\end{equation}
Given $l\geq 1$, let $P_{l,k}(t)=\sum_{j\geq 1}p_{j,k}a_{l,j}e^{i\lambda_jt}$, $k=1,2,\ldots$, be the Bochner-Fej\'{e}r polynomials which converge uniformly on $\mathbb{R}$ to $f_l$ (and converge formally to its Fourier series on $\mathbb{R}$). It is worth noting that for each $k$ only a finite number of the factors $p_{j,k}$ differ from zero, and these factors $p_{j,k}$ do not depend on $l$ \cite[p. 48]{Besi}. Thus, by taking into account (\ref{seqaddpipipi1}), it is clear that $\{P_{l,1}(t)\}_{l\geq 1}$ is a sequence of equivalent trigonometric polynomials and, by Proposition \ref{proppp},
there exists a subsequence $\{P_{l_{m,1},1}(t)\}_{m\geq 1}\subset \{P_{l,1}(t)\}_{l\geq 1}$ convergent to a certain $P_1(t)=\sum_{j\geq 1}p_{j,1}a_{j}e^{i\lambda_jt}\in \mathcal{F}_{\Lambda}$ which is in the same equivalence class as $P_{1,1}(t)$ (see also Remark \ref{unodrx}). Furthermore, by Proposition \ref{DefEquiv} and Remark \ref{add}, this means that there exist infinitely many vectors $\mathbf{x}_0^{(1)}=(x_{0,1}^{(1)},x_{0,2}^{(1)},\ldots)\in\mathbb{R}^{\sharp G_\Lambda}$ such that
\begin{equation*}
p_{j,1}a_{j}=p_{j,1}a_{1,j}e^{i<\mathbf{r}_j,\mathbf{x}_0^{(1)}>},\ j=1,2\ldots,\ \mbox{with }\lambda_j\in\Lambda.
\end{equation*}
%where $m_1$ is the number of elements of any basis for $\Lambda_1$.
%Equivalently
%\begin{equation*}%\label{seqaddpipipipi}
%a_{j}=a_{1,j}e^{i<\mathbf{r}_j,\mathbf{x}_0^{(1)}>},\ j=1,2\ldots.%,\mbox{ with }\lambda_j\in\Lambda_1.
%\end{equation*}
Analogously, from the sequence $\{P_{l_{m,1},2}(t)\}_{m\geq 1}$, we can draw a subsequence\\ $\{P_{l_{m,2},2}(t)\}_{m\geq 1}\subset \{P_{l_{m,1},2}(t)\}_{m\geq 1}$ convergent to a certain $$P_2(t)=\sum_{j\geq 1}p_{j,2}a_{j}e^{i\lambda_jt}\in \mathcal{F}_{\Lambda}$$ %where $\Lambda_2=\{\lambda_j\in\Lambda:p_{j,2}\neq 0\}\cup\Lambda_1$,
which is in the same equivalence class as $P_{1,2}(t)$. This implies that there exist infinitely many vectors $\mathbf{x}_0^{(2)}=(x_{0,1}^{(2)},x_{0,2}^{(2)},\ldots)\in\mathbb{R}^{\sharp G_\Lambda}$ such that
\begin{equation*}%\label{seqaddpipipipi2}
p_{j,2}a_{j}=p_{j,2}a_{1,j}e^{i<\mathbf{r}_j,\mathbf{x}_0^{(2)}>},\ j=1,2\ldots,\ \mbox{with } \lambda_j\in\Lambda.%,\mbox{ with }\lambda_j\in\Lambda_2,
\end{equation*}
%where $m_2$ is the number of elements of any basis for $\Lambda_2$.
%In fact, the components of the vectors $\mathbf{x}_0^{(2)}$ are subsets of the vectors $\mathbf{x}_0^{(1)}$ above (see remark \ref{add}).
In general, for each $k=2,3,\ldots$, we can extract a subsequence $\{P_{l_{m,k},k}(t)\}_{m\geq 1}$ $\subset \{P_{l_{m,k-1},k}(t)\}_{m\geq 1}$ convergent to a certain $$P_k(t)=\sum_{j\geq 1}p_{j,k}a_{j}e^{i\lambda_jt}\in \mathcal{F}_{\Lambda},$$ %where $\Lambda_k=\{\lambda_j\in\Lambda:p_{j,k}\neq 0\}\cup\Lambda_{k-1}$,
which is in the same equivalence class as $P_{1,k}(t)$ and hence  there exist infinitely many vectors $\mathbf{x}_0^{(k)}=(x_{0,1}^{(k)},x_{0,2}^{(k)},\ldots)\in\mathbb{R}^{\sharp G_\Lambda}$ %($m_k$ is the number of elements of any basis for $\Lambda_k$)
such that
\begin{equation}\label{seqaddpipipipi2}
p_{j,k}a_{j}=p_{j,k}a_{1,j}e^{i<\mathbf{r}_j,\mathbf{x}_0^{(k)}>},\ j=1,2\ldots,\ \mbox{with } \lambda_j\in\Lambda.%,\mbox{ with }\lambda_j\in\Lambda_k.
\end{equation}
%Again, the first $m_{k-1}$ components of the vectors $\mathbf{x}_0^{(k)}$ are subsets of the vectors $\mathbf{x}_0^{(k-1)}$ (see remark \ref{add}).
So we get by induction a sequence $\{P_k(t)\}_{k\geq 1}$ of trigonometric polynomials which converges formally to the series
\begin{equation}\label{bo}
\sum_{j\geq 1}a_{j}e^{i\lambda_jt},\ \mbox{with }\lambda_j\in\Lambda,
\end{equation}
and, since (\ref{seqaddpipipipi2}) is satisfied for any $k=1,2,\ldots$, we can  construct, for each integer value $n\geq 1$, a vector $\mathbf{x}_{0,n}\in\mathbb{R}^{\sharp G_{\Lambda}}$ such that, by taking into account remarks \ref{unodrx} and \ref{add}, verifies
\begin{equation*}%\label{seqaddpipipipi22}
a_{j}=a_{1,j}e^{i<\mathbf{r}_j,\mathbf{x}_{0,n}>},\ j=1,2\ldots,n\mbox{ with }\lambda_j\in\Lambda.
\end{equation*}%Por reducción al absurdo sale
Hence the series (\ref{bo}) is equivalent to $\sum_{j\geq 1}a_{1,j}e^{i\lambda_jt}$ and, by Lemma \ref{defnueva}, it is the Fourier series associated with an almost periodic function $h(t)\in AP(\mathbb{R},\mathbb{C})$ such that $h\ \shortstack{$_{{\fontsize{6}{7}\selectfont *}}$\\$\sim$}\ f_1$. Consequently, $\{P_k(t)\}_{k\geq 1}$ converges
uniformly on $\mathbb{R}$ to $h(t)\in\mathcal{G}$ and we can extract a subsequence of $\{f_l(t)\}_{l\geq1}$ which also converges uniformly on $\mathbb{R}$ to $h(t)$. %the result holds.
Indeed, take the functions $P_{l_{k,k},k}(t)$, which can be written as
$$
P_{l_{k,k},k}(t)=\sum_{j\geq 1}p_{j,k}a_{l_{k,k},j}e^{i\lambda_jt}=\sum_{j\geq 1}p_{j,k}|a_{1,j}|e^{i\alpha_{l_{k,k},j}}e^{i\lambda_jt},
$$
with $0\leq \alpha_{l_{k,k},j}\leq 2\pi$ for each $k$ and $j$.
By Helly's selection principle \cite[p. 179]{Apostol}, there exists a subsequence $\{k_m\}_{m\geq 1}\subset \{k\}_{k\geq 1}$ and a sequence of real numbers $\alpha_j$ such that $\lim_{k\to\infty}\alpha_{l_{k,k},j}=\alpha_j$ for every $j\geq 1$. Analogously, given the positive numbers $\{p_{j,k_m}\}_{m\geq 1}$ for values of $m$ sufficiently large, which are uniformly bounded (because they tend to $1$), there exists a subsequence of $\{k_m\}_{m\geq 1}$ (by abuse of language, we also denote it as $\{k_m\}_{m\geq 1}$) and a sequence of real numbers $p_j$ such that $\lim_{k\to\infty}p_{j,k_m}=p_j$ for every $j\geq 1$. Now, we use these sequences to form a new Fourier series
$\sum_{j\geq 1}p_{j}|a_{1,j}|e^{i\alpha_j}e^{i\lambda_jt}$
such that $\{P_{l_{k,k},k}(t)\}_{k\geq 1}$ converges uniformly on $\mathbb{R}$ to it. By uniqueness of the Fourier series, this means that it is the Fourier series of $h(t)$.
Finally, given $\varepsilon>0$, we can assure the existence of $k\in\mathbb{N}$ sufficiently large so that
$$|f_{l_{k,k}}(t)-h(t)|\leq |f_{l_{k,k}}(t)-P_{l_{k,k},k}(t)|+|P_{l_{k,k},k}(t)-h(t)|<\varepsilon$$
for any $t\in\mathbb{R}$, which proves that $\{f_{l_{k,k}}(t)\}$ converges uniformly on $\mathbb{R}$ to $h(t)$.
\end{proof}

As a consequence, in the topology of
uniform convergence on $\mathbb{R}$, we next prove that the family of translates of a function $f\in \mathcal{F}_{\Lambda}$ is closed on its equivalence class in $\mathcal{F}_{\Lambda}/\shortstack{$_{{\fontsize{6}{7}\selectfont *}}$\\$\sim$}$.
\begin{corollary}\label{cor5}
Let $\Lambda$ be a set of exponents and $f\in \mathcal{F}_{\Lambda}$. Then the limit points of the set of functions $\mathcal{T}_f=\{f_{\tau}(t):=f(t+\tau):\tau\in\mathbb{R}\}$ are functions which are equivalent to $f$.
\end{corollary}
\begin{proof}
The result follows easily from Lemma \ref{lem} and Proposition \ref{prop}.
\end{proof}

Now Corollary \ref{cor5} can be improved with the following result.
Indeed, we next prove that, fixed a function $f\in \mathcal{F}_{\Lambda}$, the limit points of the set of the translates $\mathcal{T}_f=\{f(t+\tau):\tau\in\mathbb{R}\}$ of $f$ are precisely the almost periodic functions which are equivalent to $f$.

\begin{theorem}\label{mth0}
Let $\Lambda$ be a set of exponents, $\mathcal{G}$ an equivalence class in $\mathcal{F}_{\Lambda}/\shortstack{$_{{\fontsize{6}{7}\selectfont *}}$\\$\sim$}$ and $f\in \mathcal{G}$. Then the set of functions $\mathcal{T}_f=\{f_{\tau}(t):=f(t+\tau):\tau\in\mathbb{R}\}$ is dense in $\mathcal{G}$.
\end{theorem}
\begin{proof}
Let $f(t)$ be a function in the class $\mathcal{F}_{\Lambda}$.
We know by Corollary \ref{cor5} that the limit points of the set of functions $\mathcal{T}_f=\{f_{\tau}(t):=f(t+\tau):\tau\in\mathbb{R}\}$ are almost periodic functions which are equivalent to $f$. We next demonstrate that
any function $h(t)$ which is equivalent to
$f(t)$ is also a limit point of $\mathcal{T}_f$.
If $\sharp\Lambda<\infty$, given $\varepsilon_n=\frac{1}{n}$, $n\in\mathbb{N}$, Corollary \ref{corol3} assures the existence of an increasing sequence $\{\tau_n\}_{n\geq 1}$ of positive real numbers such that any $n\in\mathbb{N}$ verifies
$$|f(t+\tau_n)-h(t)|<\varepsilon_n\ \forall t\in\mathbb{R}.$$
Hence the result holds for the case $\sharp\Lambda<\infty$.
Let $\sharp\Lambda=\infty$ and $\{P_n(t)\}_{n\geq 1}$, $\{Q_n(t)\}_{n\geq 1}$ the sequences of Bochner-Fej\'{e}r polynomials which converge uniformly on $\mathbb{R}$ to $f(t)$ and $h(t)$, respectively.
%-%In this case, the construction of the sequence $\{\tau_n\}_{n\geq 1}$ is also based on Corollary \ref{corol3}.
Take $$\varepsilon_1=\max\{\sup_{t\in \mathbb{R}}|f(t)-P_1(t)|,\sup_{t\in \mathbb{R}}|h(t)-Q_1(t)|\}>0,$$ then Corollary \ref{corol3} assures the existence of $\tau_1>0$ such that
$\left|P_1(t+\tau_1)-Q_1(t)\right|<\varepsilon_1,\ t\in\mathbb{R}.$
Therefore, if $t\in \mathbb{R}$ then
$$|f(t+\tau_1)-h(t)|\leq|f(t+\tau_1)-P_1(t+\tau_1)|+|P_1(t+\tau_1)-Q_1(t)|+|Q_1(t)-h(t)|\leq 3\varepsilon_1.$$ %$$\left|a_1e^{\lambda_1(s+i\tau_1)}-b_1e^{\lambda_1s}\right|+\left|\sum_{j>1}(a_je^{i\lambda_j\tau_1}-b_j)e^{\lambda_ls}\right|.$$
%Thus, from (\ref{aprop1}) and $|a_j|=|b_j|$ for any $j\geq 1$, we deduce for any $s\in K$ that
%$$|f(s+i\tau_1)-h(s)|<\varepsilon_1+2\sum_{j>1}|a_j|e^{\lambda_j\sigma_0}=3\sum_{j>1}|a_j|e^{\lambda_j\sigma_0}.$$
Similarly, take $\varepsilon_2=\max\{\sup_{t\in \mathbb{R}}|f(t)-P_2(t)|,\sup_{t\in \mathbb{R}}|h(t)-Q_2(t)|\}>0$, then Corollary \ref{corol3} assures the existence of $\tau_2>\tau_1$ such that
$\left|P_2(t+\tau_2)-Q_2(t)\right|<\varepsilon_2,\ t\in\mathbb{R}.$
Therefore, if $t\in \mathbb{R}$ then
$$|f(t+\tau_2)-h(t)|\leq|f(t+\tau_2)-P_2(t+\tau_2)|+|P_2(t+\tau_2)-Q_2(t)|+|Q_2(t)-h(t)|\leq 3\varepsilon_2.$$
In general, by repeating this process, we can construct an increasing sequence $\{\tau_n\}_{n\geq 1}$ such that each $\tau_n$ satisfies that
\begin{equation}\label{apropn111}
\left|P_n(t+\tau_n)-Q_n(t)\right|<\varepsilon_n,
\end{equation}
with $\varepsilon_n=\max\{\sup_{t\in \mathbb{R}}|f(t)-P_n(t)|,\sup_{t\in \mathbb{R}}|h(t)-Q_n(t)|\}$. Thus, from (\ref{apropn111}), any $t\in \mathbb{R}$ verifies
\begin{multline*}
$$|f(t+\tau_n)-h(t)|\leq|f(t+\tau_n)-P_n(t+\tau_n)|+ \\
+|P_n(t+\tau_n)-Q_n(t)|+|Q_n(t)-h(t)|\leq 3\varepsilon_n.$$
\end{multline*}
Note that $\varepsilon_n$ tends to $0$ when $n$ goes to $\infty$. Consequently, the sequence of functions
$\{f(t+\tau_n)\}_{n\geq 1}$ converges to $h(t)$ uniformly on $\mathbb{R}$ and the result holds.
\end{proof}

\begin{corollary}\label{cmth}
Let $f\in AP(\mathbb{R},\mathbb{C})$ and $f_1\ \shortstack{$_{{\fontsize{6}{7}\selectfont *}}$\\$\sim$}\ f$.
There exists
an increasing\-
unbounded sequence $\{\tau_n\}_{n\geq 1}$
of positive numbers
such that the sequence of functions
$\{f(t+\tau_n)\}_{n\geq 1}$ converges uniformly on $\mathbb{R}$ to $f_1(t)$.
In fact, given $\varepsilon>0$ there exists a relatively dense set of positive numbers $\tau$ such that
$$|f(t+\tau)-f_1(t)|<\varepsilon,\ \forall t\in\mathbb{R}.$$
\end{corollary}
\begin{proof}
Let $f$ be an almost periodic in $AP(\mathbb{R},\mathbb{C})$, then $f\in \mathcal{F}_{\Lambda}$ for some set $\Lambda$ of exponents. Let $\mathcal{G}$ be the equivalence class in $\mathcal{F}_{\Lambda}/\shortstack{$_{{\fontsize{6}{7}\selectfont *}}$\\$\sim$}$ such that $f\in\mathcal{G}$ and let $f_1\ \shortstack{$_{{\fontsize{6}{7}\selectfont *}}$\\$\sim$}\ f$. Thus, by Theorem \ref{mth0} (see also its proof), there exists an increasing unbounded sequence $\{\delta_n\}_{n\geq 1}$
of positive numbers such that the sequence of functions
$\{f(t+\delta_n)\}_{n\geq 1}$ converges uniformly to $f_1(t)$ on
$\mathbb{R}$.
Equivalently, given $\varepsilon>0$ there exists $n_0\in\mathbb{N}$ such that
$$|f(t+\delta_n)-f_1(t)|<\varepsilon/2\ \ \forall n\geq n_0,\ \forall t\in\mathbb{R}.$$
Moreover, since $f(t)$ is almost periodic, there exists $l=l(\varepsilon)>0$ such that any interval $(a,a+l)$ contains a number $\tau$ satisfying $|f(t+\tau)-f(t)|<\varepsilon/2$ $\forall t\in\mathbb{R}$. Hence any interval $(a,a+l)$ contains a number $\tau$ satisfying
\begin{multline*}
|f(t+\delta_n+\tau)-f_1(t)|\leq |f(t+\delta_n+\tau)-f(t+\delta_n)|+\\+|f(t+\delta_n)-f_1(t)|<\varepsilon\ \forall n\geq n_0,\ \forall t\in\mathbb{R},
\end{multline*}
which proves the result.
\end{proof}

As it was said in the introduction, Bochner's property consisting of the relative compactness of the set $\{f(t+\tau)\}$, $\tau\in\mathbb{R}$, associated with a function $f$, is a characteristic feature of
almost periodicity in the sense of Bohr. In this respect, the following main theorem, formulated in terms of Bochner's result, characterizes the space of almost periodic functions $AP(\mathbb{R},\mathbb{C})$.

%, in the sense of uniform convergence on $\mathbb{R}$,
\begin{theorem}\label{mthh}
 Let $f\in B(\mathbb{R},\mathbb{C})$. Then $f$ is in $AP(\mathbb{R},\mathbb{C})$ if and only if the closure of its set of translates is compact and it coincides with its equivalence class.
\end{theorem}
\begin{proof}
First of all, we recall that any function $f\in B(\mathbb{R},\mathbb{C})$ has an associated Fourier series.
Let $f\in AP(\mathbb{R},\mathbb{C})$, then $f\in \mathcal{F}_{\Lambda}$ for some set $\Lambda$ of exponents. Now, let $\mathcal{G}$ be the equivalence class in $\mathcal{F}_{\Lambda}/\shortstack{$_{{\fontsize{6}{7}\selectfont *}}$\\$\sim$}$ such that $f\in\mathcal{G}$.
%Take now a sequence
% $\{f(s+i\tau_n)\}$, $\tau_n\in\mathbb{R}$, of
%vertical translations of $f$, which are in $\mathcal{F}$ (see Lemma \ref{lem}), and apply Theorem \ref{prop}. Thus there exists a subsequence that converges uniformly on the compact subsets of $U$ to a function which is in $\mathcal{F}$. In fact,
By Theorem \ref{mth0}, all the limit points of the translates of $f$ are functions which are included in $\mathcal{G}$ and, in fact, the compact closure of the set of the translates of $f$ coincides with $\mathcal{G}$. Conversely, since the set of translates of a function $f$ is relatively compact, we have that $f\in AP(\mathbb{R},\mathbb{C})$ in virtue of Bochner's result.
\end{proof}

Equivalently, the theorem above shows that, for a $f\in B(\mathbb{R},\mathbb{C})$, the condition of
almost periodicity in the sense of Bohr is equivalent to that every sequence $\{f(t+\tau_n)\}$, $\tau_n\in\mathbb{R}$, of translates of $f$ has a subsequence that converges uniformly on $\mathbb{R}$ to a function which is equivalent to $f$.

\begin{corollary}
If $f\in AP(\mathbb{R},\mathbb{C})$, then the compact closure of its set of translates coincides with its equivalence class.
\end{corollary}

\subsection{On the spaces $AP(U,\mathbb{C})$, $U\subset \mathbb{C}$}

\ \\[-0.05cm]

We next deal with the case of the classes $\mathcal{D}_{\Lambda}$ (see Definition \ref{DF}) which give rise to the space of the almost periodic functions $AP(U,\mathbb{C})$ in some open vertical strips $U\subset\mathbb{C}$. It is worth remarking that the Dirichlet series associated with a function $f\in \mathcal{D}_{\Lambda}$ represents the Fourier series of $f(s)$ on any line of the strip of almost periodicity. In fact, any function $f\in \mathcal{D}_{\Lambda}$, which is almost periodic in $U=\{\sigma+it\in\mathbb{C}:\alpha<\sigma<\beta\}$, satisfies that the Fourier series of $f_{\sigma_0}(t):=f(\sigma_0+it)$, with $\sigma_0\in(\alpha,\beta)$, has the same expression $\sum_{j\geq 1}a_je^{\lambda_j\sigma_0}e^{i\lambda_j t}$ independently of $\sigma_0$.

From now on, we will consider that the space of all analytic functions on open vertical strips $U$ is endowed with the topology of uniform convergence on every reduced strip of $U$, in particular in the space of functions which have an associated Dirichlet series.
Note firstly that, by using Lemma \ref{defnueva2}, if $f\in AP(U,\mathbb{C})$ then any function of its equivalence class is also included in $AP(U,\mathbb{C})$. %$\mathcal{G}$ is almost periodic on the same open vertical strip $U$.
As in Proposition \ref{prop}, it is verified that %, with the topology of uniform convergence,
the equivalence classes of $\mathcal{D}_{\Lambda}/\shortstack{$_{{\fontsize{6}{7}\selectfont *}}$\\$\sim$}$ are closed.

\begin{proposition}\label{propppult}
Let $\Lambda$ be a set of exponents and $\mathcal{G}$ an equivalence class in $\mathcal{D}_{\Lambda}/\shortstack{$_{{\fontsize{6}{7}\selectfont *}}$\\$\sim$}$, whose functions are almost periodic in an open vertical strip $U$. In the space of analytic functions on $U$, endowed with the topology of uniform convergence on reduced strips, $\mathcal{G}$ is sequentially compact.
\end{proposition}
\begin{proof}
The proof of this result is analogous to that of Proposition \ref{prop}.
\end{proof}

In particular, we deduce from the result above that the limit points of the family of translates $\mathcal{T}_f=\{f_{\tau}(s):=f(s+i\tau):\tau\in\mathbb{R}\}$, of a function $f\in \mathcal{D}_{\Lambda}$, are included in its equivalence class in $\mathcal{D}_{\Lambda}/\shortstack{$_{{\fontsize{6}{7}\selectfont *}}$\\$\sim$}$.
More so, fixed a function $f\in \mathcal{D}_{\Lambda}$,  the limit points of the set of the translates $\mathcal{T}_f=\{f(s+i\tau):\tau\in\mathbb{R}\}$ of $f$ are precisely the functions which are equivalent to $f$, which is technically proved \textit{mutatis mutandis} as in Theorem \ref{mth0}.

\begin{theorem}\label{mth}
Let $\Lambda$ be a set of exponents, $\mathcal{G}$ an equivalence class in $\mathcal{D}_{\Lambda}/\shortstack{$_{{\fontsize{6}{7}\selectfont *}}$\\$\sim$}$
and $f\in \mathcal{G}$. %In the space of analytic functions on $U$, endowed with the topology of uniform convergence on compact subsets,
Then the set of functions $\mathcal{T}_f=\{f_{\tau}(s):=f(s+i\tau):\tau\in\mathbb{R}\}$ is dense in $\mathcal{G}$.
\end{theorem}

\begin{corollary}\label{cmth}
Let $f$ be an almost periodic function in an open vertical strip $U$ and $f_1\ \shortstack{$_{{\fontsize{6}{7}\selectfont *}}$\\$\sim$}\ f$.
There exists an increasing\-
unbounded sequence $\{\tau_n\}_{n\geq 1}$
of positive numbers such that the sequence of functions
$\{f(s+i\tau_n)\}_{n\geq 1}$ converges uniformly on every reduced strip of $U$ to $f_1(s)$. In fact, given $\varepsilon>0$ and a reduced strip $U_1\subset U$ there exists a relatively dense set of positive numbers $\tau$ such that
$$|f(s+i\tau)-f_1(s)|<\varepsilon,\ \forall s\in U_1.$$
\end{corollary}

Finally, the following main theorem, formulated in terms of Bochner's result, characterizes the space of functions $AP(U,\mathbb{C})$. Its demonstration is analogous to that of Theorem \ref{mthh}.

\begin{theorem}\label{mthhcc}Let $f\in B(U,\mathbb{C})$. Then $f$ is in $AP(U,\mathbb{C})$ if and only if the closure of its set of vertical translates is compact and it coincides with a certain equivalence class of $\mathcal{D}_{\Lambda}/\shortstack{$_{{\fontsize{6}{7}\selectfont *}}$\\$\sim$}$ for some set $\Lambda$ of exponents.
\end{theorem}

\section{Some applications to the exponential sums which converge absolutely and, in particular, to the Riemann zeta function}\label{ms}

The results of the previous section can be particularized to the following classes of exponential sums.
\begin{definition}
Let $\Lambda=\{\lambda_1,\lambda_2,\ldots,\lambda_j,\ldots\}$ be a set of exponents. We will say that a function $f:\mathbb{C}\mapsto\mathbb{C}$ is in the class $\mathcal{A}_{\Lambda}$ if it is an exponential sum of the form
 \begin{equation}\label{eqq}
f(s)=\sum_{j\geq 1}a_je^{\lambda_js},\ a_j\in\mathbb{C},\ \lambda_j\in\Lambda,
\end{equation}
where the sum appearing in the right hand side of (\ref{eqq}) converges absolutely on some non-empty set $U\subset\mathbb{C}$.
%i.e., the set of exponents of these sums is $\Lambda$.
\end{definition}

For a given set $\Lambda$ of exponents, the classes $\mathcal{A}_{\Lambda}$ are clearly non-empty and, in fact, it is easy to demonstrate that any $f\in \mathcal{A}_{\Lambda}$ converges absolutely on some vertical strip $U=\{s=\sigma+it\in\mathbb{C}:\sigma_l<\sigma<\sigma_r\}$, where $\sigma_l$ and $\sigma_r$ could eventually be $-\infty$ and $\infty$ respectively.

\begin{remark}\label{remarkk}
Let $AP_1(\mathbb{R},\mathbb{C})$ be the space of the almost periodic functions $h:\mathbb{R}\mapsto\mathbb{C}$ whose Fourier series are absolutely convergent, i.e. functions of the form $h(t)=\Sum_{j\geq 1}a_je^{i\lambda_j t}$, $t\in\mathbb{R}$, $\lambda_j\in\mathbb{R}$, $a_j\in\mathbb{C}$ (see \cite[Section 4.2]{Corduneanu}). Fixed a set $\Lambda$ of exponents, if $f(s)\in \mathcal{A}_{\Lambda}$ is of type (\ref{eqq}), which converges absolutely on a region $U=\{s=\sigma+it:\sigma_l<\sigma<\sigma_r\}$, then
the function $h_{\sigma_0}:\mathbb{R}\mapsto\mathbb{C}$, with $\sigma_0\in (\sigma_l,\sigma_r)$, defined as
$h_{\sigma_0}(t):=f(\sigma_0+it),\ t\in\mathbb{R},$
is in $AP_1(\mathbb{R},\mathbb{C})$ and it is also uniformly convergent.
Moreover, the function $f(s)$ is almost periodic in the strip $U$ \cite[p.144]{Besi} and it coincides with its associated Dirichlet series \cite[p.148, Theorem 1]{Besi}.
More so, the family $\{\mathcal{A}_{\Lambda}:\Lambda\ \mbox{is a set of exponents}\}$
gives us the space of the almost periodic functions whose Dirichlet series are absolutely convergent on some open vertical strip $U\subset \mathbb{C}$.
\end{remark}

\begin{remark}\label{dewfc}
Given $\Lambda=\{\lambda_1,\lambda_2,\ldots,\lambda_j,\ldots\}$ a set of exponents, let $f_1(s)=\sum_{j\geq 1}a_je^{\lambda_js}$ and $f_2(s)=\sum_{j\geq 1}b_je^{\lambda_js}$ be two exponential sums in $\mathcal{A}_{\Lambda}$ which converge absolutely on $U_{f_1}$ and $U_{f_2}$ respectively. If $f_1\ \shortstack{$_{{\fontsize{6}{7}\selectfont *}}$\\$\sim$}\ f_2$, it is clear that $|a_j|=|b_j|$ for each $j\geq 1$ and thus $U_{f_1}=U_{f_2}$. Hence, if $f_1\in AP(U,\mathbb{C})$, its Dirichlet series converges absolutely on $U$ and $f_1\ \shortstack{$_{{\fontsize{6}{7}\selectfont *}}$\\$\sim$}\ f_2$, then $f_2\in AP(U,\mathbb{C})$ and its Dirichlet series converges absolutely on $U$.
\end{remark}
From the remark above, with the topology of the uniform convergence on every reduced strip in $U$, we have the following result as a corollary of Theorem \ref{mth}.

\begin{corollary}\label{mthco}
Let $\Lambda$ be a set of exponents, $\mathcal{G}$ an equivalence class in $\mathcal{A}_{\Lambda}/\shortstack{$_{{\fontsize{6}{7}\selectfont *}}$\\$\sim$}$, whose functions converge absolutely on an open vertical strip $U$, and $f\in \mathcal{G}$.
%In the space of analytic functions on $U$, endowed with the topology of uniform convergence on compact subsets,
Then the set of functions $\mathcal{T}_f=\{f_{\tau}(s):=f(s+i\tau):\tau\in\mathbb{R}\}$ is dense in $\mathcal{G}$.
\end{corollary}

Let $\Lambda_P=\{-\log 2,-\log 3,-\log 4,\ldots,-\log j,\ldots\},$ then the Riemann zeta function
$\zeta(s)=\sum_{n=1}^{\infty}e^{-s\log n}=\sum_{n\geq 1}\frac{1}{n^s}$,
which converges absolutely in $\{s=\sigma+it:\sigma>1\}$, is in the class $\mathcal{A}_{\Lambda_P}$.
In fact, a basis for $\Lambda_P$ is given by
$G_{\Lambda_P}=\{-\log 2,-\log3,-\log5,\ldots,-\log p_k,\ldots\},$ where $p_k$ is the $k$-th prime number.
%We next obtain interesting consequences of the above results which, at least to our knowledge, have not been considered in the literature. In particular,
Furthermore, from Corollary \ref{cmth}, we get the next consequence for the Riemann zeta function $\zeta(s)$.

\begin{theorem}\label{mtrzf}
Let $f_1(s)=\sum_{n\geq 1}\frac{a_n}{n^s}$, $a_n\in\mathbb{C}$, be an exponential sum which is equivalent to $\zeta(s)$.
There exists an increasing unbounded sequence $\{\tau_n\}_{n\geq 1}$
of positive numbers such that the sequence of functions
$\{\zeta(s+i\tau_n)\}_{n\geq 1}$ converges uniformly to $f_1(s)$ on
every reduced strip of $\{s\in\mathbb{C}:\operatorname{Re}s>1\}$.
\end{theorem}

Let $\zeta_{\lambda}(s):=\sum_{n\geq 1}\frac{\lambda(n)}{n^s}$ be the Dirichlet series for the Liouville function $\lambda(n)$ \cite{Lehman}. Since $\zeta(s)$ is equivalent to itself and it is also equivalent to $\zeta_{\lambda}(s)$, we get the following corollary.

\begin{corollary}\label{prob1}
There exist two increasing unbounded sequences $\{\tau_n\}_{n\geq 1}$ and $\{\varsigma_n\}_{n\geq 1}$
of positive numbers such that the sequences of functions
$\{\zeta(s+i\tau_n)\}_{n\geq 1}$ and $\{\zeta(s+i\varsigma_n)\}_{n\geq 1}$ converge uniformly to $\zeta(s)$ and $\zeta_{\lambda}(s)$, respectively, on
every reduced strip of $\{s\in\mathbb{C}:\operatorname{Re}s>1\}$.
\end{corollary}

As a consequence, we will obtain alternative demonstrations of some known
results related to the infimum of $|\zeta(s)|$ on certain regions in the half-plane $\{s=\sigma+it:\sigma\geq 1\}$.
We first consider the following preliminary result (compare also with \cite[Theorem
8.7]{Titchmarsh} and \cite[p. 288, Theorem 2]{KaraVoro}).

\begin{lemma}\label{prop1}
Let $\sigma_{0}$ be a real number greater than $1$. Then
\begin{equation*}%\label{infimum}
\inf \{\left\vert \zeta (s)\right\vert
:\operatorname{Re}s\geq\sigma _{0}\}=\frac{\zeta(2\sigma_0)}{\zeta(\sigma_0)} =\prod_{k=1}^{\infty}\frac{1}{1+p_{k}^{-\sigma _{0}}%
}.
\end{equation*}
Moreover, $\left\vert \zeta
(s)\right\vert>\Prod_{k=1}^{\infty}\frac{1}{1+p_{k}^{-\sigma _{0}}}$ $\forall
s\in\mathbb{C}:\operatorname{Re}s\geq\sigma _{0}$.
\end{lemma}
\begin{proof} If $\sigma_{0}>1$, it was proved in \cite[Section
7.6]{Apostol} that
$$\inf \{\left\vert \zeta (s)\right\vert
:\operatorname{Re}s=\sigma
_{0}\}=\Frac{\zeta(2\sigma_0)}{\zeta(\sigma_0)}.$$
On
the other hand, by using the Euler product formula, we have
$$\frac{\zeta(2\sigma_0)}{\zeta(\sigma_0)}=\lim_{n\rightarrow\infty}\frac{\Prod_{k=1}^{n}\Frac{1}{1-p_k^{-2\sigma_0}}}{\Prod_{k=1}^{n}\Frac{1}{1-p_k^{-\sigma_0}}}=\lim_{n\rightarrow\infty}\prod_{k=1}^{n}\frac{1}{1+p_{k}^{-\sigma _{0}}}=\prod_{k=1}^{\infty}\frac{1}{1+p_{k}^{-\sigma_{0}}},$$
%$$\zeta(\sigma_0)\cdot\prod_{k=1}^{\infty}\frac{1}{1+p_{k}^{-\sigma _{0}}}=\Prod_{k=1}^{\infty}\Frac{1}{1-p_k^{-\sigma_0}}\cdot \prod_{k=1}^{\infty}\frac{1}{1+p_{k}^{-\sigma _{0}}}=\Prod_{k=1}^{\infty}\Frac{1}{1-p_k^{-2\sigma_0}}=\zeta(2\sigma_0),$$
which means that
$\inf \{\left\vert \zeta (s)\right\vert
:\operatorname{Re}s=\sigma
_{0}\}=\prod_{k=1}^{\infty}\frac{1}{1+p_{k}^{-\sigma _{0}}}.$
Now, if we take $\sigma\geq \sigma_0$ observe that
$$\Frac{1}{1+p_{k}^{-\sigma}}\geq \Frac{1}{1+p_{k}^{-\sigma_0}}$$
for each $k=1,2,\ldots$. Hence
$$\Prod_{k=1}^{\infty}\frac{1}{1+p_{k}^{-\sigma}}\geq
\Prod_{k=1}^{\infty}\frac{1}{1+p_{k}^{-\sigma _{0}}}$$ and
consequently
\begin{equation*}%\label{infimum}
\inf \{\left\vert \zeta (s)\right\vert
:\operatorname{Re}s\geq\sigma _{0}\}=\inf \{\left\vert \zeta
(s)\right\vert
:\operatorname{Re}s=\sigma _{0}\}=\prod_{k=1}^{\infty}\frac{1}{1+p_{k}^{-\sigma _{0}}%
}.
\end{equation*}
Finally, if we suppose the existence of some $s_1=\sigma_0+it_1$,
$t_1\in\mathbb{R}$, such that $\left\vert \zeta (s_1)\right\vert
=\prod_{k=1}^{\infty}\frac{1}{1+p_{k}^{-\sigma _{0}}%
}$, we deduce from the Euler product
formula that
\begin{equation}\label{rec}
\vert\zeta(s_1)\vert=\Prod_{k=1}^{\infty}\left\vert\Frac{1}{1-p_k^{-s_1}}\right\vert=\prod_{k=1}^{\infty}\Frac{1}{1+p_k^{-\sigma_0}}.
\end{equation}
Note that, for each $k=1,2,\ldots$, we have
$$\left\vert\Frac{1}{1-p_k^{-s_1}}\right\vert=\left|\frac{1}{1-p_k^{-\sigma_0}e^{-it_1\log p_k}}\right|\geq \Frac{1}{1+p_k^{-\sigma_0}}$$
and the equality is verified if and only if
$t_1=\Frac{(2m_k+1)\pi}{\log p_k}$ for some $m_k\in\mathbb{Z}$.
Therefore, from (\ref{rec}), this implies for each $k=1,2,\ldots$
that
$t_1=\Frac{(2m_k+1)\pi}{\log p_k}$ for some $m_k\in\mathbb{Z},$
which is clearly a contradiction because the numbers $\{\log
p_k:k=1,2,\ldots\}$ are linearly independent over the rationals.
\end{proof}
\vspace{1cm}

\begin{remark}\label{cor2}
We recall that the Riemann zeta function $\zeta(s)$ has a simple pole at $s=1$. Therefore,
since $\Frac{\zeta(2\sigma)}{\zeta(\sigma)}$ goes to $0$ as
$\sigma$ tends to $1^+$, we immediately obtain from Lemma \ref{prop1} that
\begin{equation*}%<-- buscar dónde está esto. Me suena en Turán
\inf \{\left\vert \zeta (s)\right\vert :\operatorname{Re}s> 1\}=\inf \{\left\vert \zeta (s)\right\vert :1<\operatorname{Re}s<\sigma_0\}=0,
\end{equation*}
for any $\sigma_0>1$.
That is, in spite of the absence of zeros, it is well known that $|\zeta(s)|$ takes arbitrarily small values in $\{s\in\mathbb{C}:\operatorname{Re}s>1\}$.
%We note that, answering a question of D. Hilbert, this result was first proved by H. Bohr and E. Landau.
\end{remark}

As a consequence of Corollary \ref{prob1}, we next provide another
proof of \cite[Theorem 8.6]{Titchmarsh} or \cite[p. 288, Corollary 1]{KaraVoro}. In fact, for any positive
number $\tau>0$, we next construct a sequence of complex numbers
$\{w_n\}\subset \{s\in\mathbb{C}:\operatorname{Re}s>
1,\operatorname{Im}s>\tau\}$ such that
$\Lim_{n\rightarrow\infty}\zeta(w_n)=0$.
%The next purpose in this paper is to obtain a sequence of complex numbers of real part equal to $1$ with the same property, that is $\Lim_{n\rightarrow\infty}\zeta(1+t_n^{*})=0$ for some sequence $\{t_n^*\}$ of positive numbers.

\begin{corollary}\label{positive}
Fixed $\tau>0$ and $\sigma_0>1$, $$
\inf \{\left\vert \zeta
(s)\right\vert :\operatorname{Re}s> 1,\operatorname{Im}s>\tau\}=\inf \{\left\vert \zeta
(s)\right\vert :1<\operatorname{Re}s<\sigma_0,\operatorname{Im}s>\tau\}=0.$$
\end{corollary}
\begin{proof}
Fix $\tau>0$ and $\sigma_0>1$. From Remark \ref{cor2}, it is immediate that there
exists a sequence $\{w_m\}_{m\geq
1}\subset\{s\in\mathbb{C}:1<\operatorname{Re}s<\sigma_0\}$, with
$w_m=\sigma_m+it_m$, such that $|\zeta(w_m)|<\Frac{1}{m}$ for each
$m\in\mathbb{N}$. Suppose $t_m\leq \tau$ for some integer number $m\geq 1$. Thus, by
Corollary \ref{prob1}, there exists an increasing unbounded
sequence $\{\tau_{n}\}$ of positive real numbers such that
$\left\{\zeta(s+i\tau_{n})\right\}_n$ converges uniformly to
$\zeta(s)$ on every reduced strip $U_1$ of
$\{s\in\mathbb{C}:\operatorname{Re}s>1\}$. Take a reduced strip $U_1$ such that $w_m\in U_1$. Particularly, we have
$\lim_{n\rightarrow\infty}\zeta(w_m+i\tau_{n})=\zeta(w_m)$
and, consequently, there exists $n_0\in\mathbb{N}$ such that
$|\zeta(w_m+i\tau_{n})|<\Frac{1}{m}$ for each $n\geq n_0$. Finally,
since $\{\tau_{n}\}$ is unbounded, there exists an integer number
$p_0\geq n_0$ such that $t_m+\tau_{p_0}>\tau$ and
$|\zeta(w_m+i\tau_{p_0})|<\Frac{1}{m}$. Thus we can construct a
sequence $\{w_m^*\}\subset \{s\in\mathbb{C}:1<\operatorname{Re}s<
\sigma_0,\operatorname{Im}s>\tau\}$ verifying
$|\zeta(w_m^*)|<\Frac{1}{m}$ for each $m\in\mathbb{N}$. Hence the
result holds.
\end{proof}

Finally, the following corollary is related to \cite[Theorem 8.6
(A)]{Titchmarsh}.

\begin{corollary}\label{ult}
There exists a sequence of positive numbers $\{t_n\}_{n\geq 1}$
such that\\ $\Lim_{n\rightarrow\infty}\zeta(1+it_n)=0$.
\end{corollary}
\begin{proof}
%By Lemma \ref{prop1}, $\Frac{1}{\zeta(s)}$ is bounded throughout any region
%$\{s\in\mathbb{C}:\sigma_0\leq \operatorname{Re}s\leq
%\sigma_1\},$ with $1<\sigma_0<\sigma_1$. Moreover, by Corollary \ref{positive}, $\Frac{1}{\zeta(s)}$ is unbounded throughout $\{s\in\mathbb{C}:1<\operatorname{Re}s\leq
%\sigma_1\}$. This implies that $\Frac{1}{\zeta(1+it)}$ is unbounded as
%$t\rightarrow\infty$. Otherwise, $\Frac{1}{\zeta(\sigma+it)}$ would not be continuous when $\sigma=1$.
%Thus there exists a sequence of positive
%numbers $\{t_n\}_{n\geq 1}$ such that
%$\Lim_{n\rightarrow\infty}\zeta(1+it_n)=0$.
%
Suppose that $\Frac{1}{\zeta(1+it)}$ is bounded as
$t\rightarrow\infty$. Given $\tau>0$ and $\sigma_1>1$, let $M_1=\sup\left\{\left|\Frac{1}{\zeta(1+it)}\right|:t>\tau\right\}<\infty$, $M_2=\sup\left\{\left|\Frac{1}{\zeta(\sigma_1+it)}\right|:t>\tau\right\}$ and $M_3=\max\left\{\left|\Frac{1}{\zeta(\sigma+i\tau)}\right|:0\leq \sigma\leq\sigma_1\right\}$. In this case,
from Lemma \ref{prop1} and Phragm\'{e}n and Lindel\"{o}f's theorem \cite[Chapter 5]{Titchmarsh2}, we have that
$\Frac{1}{\zeta(s)}$ is bounded throughout the region
$R=\{s\in\mathbb{C}:1\leq \operatorname{Re}s\leq
\sigma_1,\operatorname{Im}s>\tau\}$. In  fact, an admisible upper bound for the function $\Frac{1}{\zeta(s)}$ in $R$ is $M=\max\{M_1,M_2,M_3\}$. Nevertheless, by Corollary
\ref{positive}, this is false. Consequently, we
get that $\Frac{1}{\zeta(1+it)}$ is unbounded as
$t\rightarrow\infty$ and thus there exists a sequence of positive
numbers $\{t_n\}_{n\geq 1}$ such that
$\Lim_{n\rightarrow\infty}\zeta(1+it_n)=0$.
\end{proof}

Finally, we would like to add the following remarks.

\begin{remark}\label{remarkfinal}
If $\Lambda_A=\{\gamma_1,\gamma_2,\ldots,\gamma_j,\ldots\}$ is a countable set of distinct aligned frequencies or exponents (not necessarily real numbers), then the main results in this paper can be easily extended to the exponential sums of the form
$$\sum_{j\geq 1}a_je^{\gamma_js},\ a_j\in\mathbb{C},\ \gamma_j\in\Lambda_A.$$
%where the sum appearing in the right hand side converges uniformly on
%some non-empty open set $U\subset \mathbb{C}$.
In this case, we must change the vertical lines, strips or translates by crosswise ones which are perpendicular to the line given by the frequencies.
\end{remark}

\begin{remark}
Given $\Lambda=\{\lambda_1,\lambda_2,\ldots,\lambda_j,\ldots\}$ a set of exponents, consider $A_1(p)$ and $A_2(p)$ two exponential sums in the class $\mathcal{S}_{\Lambda}$, say
$A_1(p)=\sum_{j\geq1}a_je^{\lambda_jp}$ and $A_2(p)=\sum_{j\geq1}b_je^{\lambda_jp}.$
We will say that $A_1$ is Bohr-equivalent to $A_2$ if $a\sim b$, where $a,b:\Lambda\to\mathbb{C}$ are the functions given by $a(\lambda_j):=a_j$ y $b(\lambda_j):=b_j$, $j=1,2,\ldots$ and $\sim$ is in Definition \ref{DefEquiv0}. Now, fixed a basis $G_{\Lambda}$ for $\Lambda$, for each $j\geq1$ let $\mathbf{r}_j$ be the vector of rational components verifying (\ref{errej}).  %-%$$\lambda_j=<\mathbf{r}_j,\mathbf{g}>=\sum_{k=1}^{q_j}r_{j,k}g_k,$$ where
  %$\mathbf{g}=(g_1,g_2,\ldots,g_k,\ldots)$ is the vector of the elements of a basis $G_{\Lambda}$ for $\Lambda$.
Thus it is easy to prove that $A_1$ is Bohr-equivalent to $A_2$ %, relative to the basis $G_{\Lambda}$,
if and only if there exists $\mathbf{x}_0=(x_{0,1},x_{0,2},\ldots,x_{0,k},\ldots)\in \mathbb{R}^{\sharp G_{\Lambda}}$
such that $b_j=a_j e^{<\mathbf{r}_j,\mathbf{x}_0>i}$ for every $j\geq 1$. %In that case, we will write $A_1\sim A_2$.

From this and Proposition \ref{DefEquiv}, it is worth noting that Definition \ref{DefEquiv00} and definition of Bohr-equivalence are equivalent in the case that it is possible to obtain an integral basis for the set of exponents $\Lambda$. Consequently, all the results of this paper which can be formulated in terms of an integral basis are also valid under the Bohr-equivalence (in particular, those related to the finite exponential sums in Section \ref{section3} and the Riemann zeta function in Section \ref{ms}).
\end{remark}
%%-------------------------------

\noindent \textbf{Acknowledgements.} The authors thank the anonymous referee for its valuable comments on our manuscript in The Ramanujan Journal which led us to generalize our results. The first author's research was partially supported by Generalitat Valenciana under project GV/2015/035.

\bibliographystyle{amsplain}

\end{document}